\documentclass[12pt]{article}
\usepackage{amsmath,amsfonts,amssymb,mathtools}
\usepackage{amsthm} 
\usepackage{amsmath}
\usepackage[hmargin=2.5cm,vmargin=3cm,centering]{geometry}
\usepackage{graphicx}
\usepackage{hyperref}
\hypersetup{
    colorlinks,
    citecolor=black,
    filecolor=magenta,
    linkcolor=black,
    urlcolor=black,
    }
\usepackage[titletoc]{appendix}
\usepackage{color}
 \definecolor{red}{rgb}{1,0,0}

\newcommand{\xddots}{%
  \raise 4pt \hbox {.}
  \mkern 6mu
  \raise 1pt \hbox {.}
  \mkern 6mu
  \raise -2pt \hbox {.}
}

 \theoremstyle{break}
\newtheorem{thm}{Theorem}[section] 
\newtheorem{RQ}[thm]{Remark}

\newtheorem{NTS}[thm]{Notations}
\newtheorem{prop}[thm] {Proposition}
\newtheorem{cor}[thm] {Corollary}
\newtheorem{defi}[thm]{Definition}
\newtheorem{defis}[thm]{Definitions}

\numberwithin{equation}{section}
\newtheorem{lem}[thm]{Lemma}
\newtheorem{NT}[thm]{Notation}
\newtheorem{assumption}{Assumption}

\usepackage{float}

\title{\large\bf{Global subelliptic estimates
for Kramers-Fokker-Planck
 operators with some class of polynomials}}

\author{
   Mona Ben Said\\
Laboratoire Analyse, G{\'e}om{\'e}trie et Applications\\
   Universit{\'e} Paris 13\\
   99 Avenue Jean Baptiste Cl{\'e}ment \\93430 Villetaneuse, France\\
  bensaid@univ-paris13.fr 
   }

\begin{document}
\maketitle

\begin{abstract}
In this article we study some Kramers-Fokker-Planck operators with a polynomial potential $V(q)$ of degree greater than two having quadratic limiting behavior. This work provides an accurate global subelliptic estimate for KFP operators under some conditions imposed on the potential.
\end{abstract}

\noindent\textbf{MSC-2010:} 35Q84, 35H20, 35P05, 47A10, 14P10
\tableofcontents
\section{Introduction and main results}
The Kramers-Fokker-Planck operator reads
\begin{align}
K_V=p.\partial_q-\partial_qV(q).\partial_p+\frac{1}{2}(-\Delta_p+p^2)~,\;\;\;\;\;(q,p)\in
  \mathbb{R}^{2d}
\,,
\label{eq1}
\end{align}
where $q$ denotes the space variable, $p$ denotes the velocity
variable, $x.y=\sum\limits_{j=1}^{d}x_{j}y_{j}$\,,
  $x^{2}=\sum\limits_{j=1}^{d}x_{j}^{2}$ and the potential $V(q) = \sum\limits_{|\alpha|\le r}V_{\alpha}q^{\alpha}$
is a real-valued polynomial function on $\mathbb{R}^d$ with
$d^{\circ}V = r.$
  
There have been several works concerned with the operator $K_V$ with diversified approaches. In this article we impose some kind of assumptions on the polynomial potential $V (q),$ so that the Kramers-Fokker-Planck operator $K_V$ admits a global subelliptic estimate and has a compact resolvent. This problem is closely related to the return to the equilibrium for the Kramers-Fokker-Planck operator (see \cite{HeNi}\cite{Nie}\cite{Nou}).  
As mentioned in \cite{HerNi} and \cite{Nie}, the analysis of $K_{V}$ is also strongly linked to the one of the Witten Laplacian 
$\Delta_{V}^{(0)}=-\Delta_{q}+\left|\nabla V(q)\right|^{2}-\Delta V(q).$ This relation yielded to the following conjecture established by Helffer-Nier: 
\begin{align}
(1+K_V)^{-1}\;\text{compact} \Leftrightarrow (1+\Delta_{V}^{(0)})^{-1}\;\text{compact}~.\label{con}
\end{align}
This conjecture has been partially resolved in simple cases (see for example \cite{HeNi}, \cite{HerNi} and \cite{Li}), whereas for the operator $\Delta_{V}^{(0)}$ very general criteria
of compactness work for polynomial potiential $V(q)$ of arbitrary degree.
These last criteria require an analysis of the degeneracies at infinity of the potential and rely on extremely sophisticated tools
of hypoellipticity developed by Helffer and Nourrigat in the 1980's (see \cite{HeNo}, \cite{Nie}). Among the particularities of these last analysis, we mention that the  compactness results obtained for degenerate potentials at infiniy were not the same for $\Delta_{+V}^{(0)}$ as $\Delta_{-V}^{(0)}.$ The typical example which was considered is the case $V(q_1, q_2) = q_1^2q_2^2 $ in dimension $d = 2$: The operator $ \Delta_{- V}^{(0)} $ has a compact resolvent, while $\Delta_{+V}^{(0)}$ has not.

In the case of the Kramers-Fokker-Planck operator, there have been extensive works concerned with the case $d^{\circ}V\le2$ (see  \cite{Hor}\cite{HiPr}\cite{Vio}\cite{Vio1}\cite{AlVi}\cite{BNV}). Nevertheless,  as far as general potential is concerned, different kind of sufficient conditions on $V(q)$ had been examined by H{\'e}rau-Nier \cite{HerNi}, Helffer-Nier \cite{HeNi}, Villani \cite{Vil} and Wei-Xi Li \cite{Li}. These first results considered only variants of the elliptic situation at the infinity ( for non-degenerate potential), which did not distinguish the sign $\pm V(q).$  Lately a significant improvement of those works has been done by Wei-Xi Li \cite{Li2} based on some  multipliers methods. In \cite{Li2}, Wei-Xi Li showed that for potentials similar to $ V(q_1,q_2) = q_1^2q_2^2 $ the results for $ K_{\pm V}$ were the same as for $\Delta_{\pm V}^{(0)},$ thus comforting the idea that the conjecture (\ref{con}) is true. 

The ultimate goal would be to develop a complete recurrence with respect to $d^{\circ}V$ for the Kramers-Fokker-Planck operator like it is possible to do for the Witten Laplacian as recalled in \cite{HeNi} (cf. Teorem 10.16 page 106) and \cite{Nie} by following the general approach of Helffer-Nourrigat in \cite{HeNo} and \cite{Nou}. Although we are not able to write a complete induction, we establish here subelliptic estimates for $K_{V}$ for a rather general class of polynomial potentials with criteria which distinguish clearly the sign $\pm V(q)$. The asymtotic behaviour of those polynomials is governed by at most quadratic parameter dependent potentials, and the global subelliptic estimates in which arise some logarithmic weights are know to be essentially optimal in the quadratic case (see \cite{BNV}).

Denoting 
\[
	O_p=\frac{1}{2}(D^2_p+p^2)
\;,\]
and
\[
X_V=p.\partial_q-\partial_qV(q).\partial_p~,
\] 
we can rewrite the Kramers-Fokker-Planck operator $K_V$ defined in (\ref{eq1}) as
$K_V=X_V+O_p\;.$ \\
\textbf{Notations:} 
Throughout the paper we use the notation
\begin{align*}
\langle \cdot\rangle=\sqrt{1+|\cdot|^2}\;.
\end{align*}
For an arbitrary polynomial $V(q)$ of degree $r$, we denote for all $q\in\mathbb{R}^d$ 
\[
	\begin{aligned}
	\mathrm{Tr}_{+,V}(q) & = \sum\limits_{\substack{\nu\in \mathrm{Spec}(\mathrm{Hess}\; V)\\ \nu>0}} \nu(q)\,,
	\\ \mathrm{Tr}_{-,V}(q) &=-\sum\limits_{\substack{\nu\in \mathrm{Spec}(\mathrm{Hess}\; V)\\ \nu\le 0}}\nu(q)\;.
	\end{aligned}
\]
Futhermore, for a polynomial $P\in  E_r:=\left\{P\in \mathbb{R}[X_1,...,X_d],\;  d^\circ P\le r\right\}$ and all natural number $n\in\left\{1,...,r\right\}$, we define the functions $R^{^{\ge n}}_{P}:\mathbb{R}^d\to \mathbb{R}$ and $R^{= n}_{P}:\mathbb{R}^d\to \mathbb{R}$ by
 \begin{align}
R^{^{\ge n}}_P(q)=\sum\limits_{n\le \left|\alpha\right|\le r}\left|\partial^{\alpha}_qP(q)\right|^{\frac{1}{\left|\alpha\right|}}\;,\label{A.33}
 \end{align}
 \begin{align}
R^{= n}_P(q)=\sum\limits_{ |\alpha|=n}|\partial^{\alpha}_qP(q)|^{\frac{1}{|\alpha|}}\;.
 \end{align}
For arbitrary real functions $A$ and $B$, we make also use of the following notation \begin{align*}
 A \asymp B \iff \exists\, c\geq 1 : c^{-1}\left|B\right| \leq \left|A\right| \leq c\left|B\right|~.
 \end{align*}
This work is essentially based on the recent publication by Ben Said,
Nier, and Viola \cite{BNV}, which deals with the analysis of Kramers-Fokker-Planck operators with polynomials of degree less than 3. In this case we define the constants $A_V$ and $B_V$ by 
\begin{align*}
A_V &= \max \{(1+\mathrm{Tr}_{+,V})^{2/3}, 1+\mathrm{Tr}_{-,V}\}~,
\,	
\\	B_V &= \max\{\min\limits_{q\in\mathbb{R}^d}\left|\nabla\;V(q)\right|^{4/3}, \frac{1+\mathrm{Tr}_{-,V}}{\log(2+\mathrm{Tr}_{-,V})^2}\}\;.
\end{align*}
As proved in \cite{BNV}, there is a constant $c>0$ such that the following global subelliptic estimate with remainder
\begin{align}
\|K_Vu\|^2_{L^2(\mathbb{R}^{2d})}+A_V\|u\|^2_{L^2(\mathbb{R}^{2d})}\ge
{c} \Big(\|O_pu\|^2_{L^2(\mathbb{R}^{2d})}&+\|X_Vu\|^2_{L^2(\mathbb{R}^{2d})}\nonumber\\&+\|\langle\partial_q V(q)\rangle^{2/3}u\|^2_{L^2(\mathbb{R}^{2d})}+\|\langle D_q\rangle^{2/3}u\|_{L^2(\mathbb{R}^{2d})}\Big)\label{eq4}
\end{align}
holds for all $u\in \mathcal{C}_0^{\infty}(\mathbb{R}^{2d}).$ Moreover, if $V$ does not have any local minimum, that is if $\mathrm{Tr}_{-,V}+\min\limits_{q\in\mathbb{R}^d}\left|\nabla\;V(q)\right|\not=0$,  there exists a constant $c>0$ such that 
\begin{align}
\|K_Vu\|^2_{L^2(\mathbb{R}^{2d})}\ge c\,B_V\|u\|^2_{L^2(\mathbb{R}^{2d})}~,\label{1.5m}
\end{align}
holds for all $u\in \mathcal{C}_0^{\infty}(\mathbb{R}^{2d}).$ Hence combining (\ref{1.5m}) and (\ref{eq4}), there is a constant $c>0$ so that 
\begin{align}
\|K_Vu\|^2_{L^2(\mathbb{R}^{2d})}\ge \frac{c}{1+\frac{A_V}{B_V}}\Big(\|O_pu\|^2_{L^2(\mathbb{R}^{2d})}&+\|X_Vu\|^2_{L^2(\mathbb{R}^{2d})}\nonumber\\&+\|\langle\partial_q V(q)\rangle^{2/3}u\|^2_{L^2(\mathbb{R}^{2d})}+\|\langle D_q\rangle^{2/3}u\|_{L^2(\mathbb{R}^{2d})}\Big)
\label{eq5}
\end{align}
is valid for all $u\in \mathcal{C}_0^{\infty}(\mathbb{R}^{2d}).$
The constants appearing in (\ref{eq4}), (\ref{1.5m}) and (\ref{eq5}) are independent of the potential $V.$  We recall here that for a smooth potential $V\in\mathcal{C}^{\infty}(\mathbb{R}^d)$, our operator $K_V$ is essential maximal accretive when endowed with the domain $\mathcal{C}_0^{\infty}(\mathbb{R}^{2d})\;\cite{HeNi}$ (cf. Proposition 5.5 page 44). As a result the domain of its closure is given by
\begin{align*}
D(K_V)=\left\{u\in L^2(\mathbb{R}^{2d}),\; K_Vu\in L^2(\mathbb{R}^{2d})\right\}~.
\end{align*}
Consequently by density of $\mathcal{C}_0^{\infty}(\mathbb{R}^{2d})$ in $D(K_V)$ all estimates stated in this paper, which are checked with $C^\infty_0(\mathbb{R}^{2d})$ functions, can be extended to the domain of $K_V.$ 

Given a polynomial $V(q)$ with degree $r$ greater than two, our result
will require the following assumption after setting for $\kappa>0$
\begin{align*}
\Sigma(\kappa)=\left\{q\in\mathbb{R}^d,\;\left|\nabla V(q)\right|^{\frac{4}{3}}\ge \kappa\Big(\left|\mathrm{Hess}\; V(q)\right|+R_{V}^{^{\ge 3}}(q)^4+1\Big)\right\}~.
\end{align*}
\begin{assumption}
\label{assumption1}
There exist large constants $\kappa_0,C_1>1$
such that for all $\kappa\ge \kappa_0$ the polynomial  $V(q)$  satisfies the following properties

\begin{align}
\mathrm{Tr}_{-,V}(q)\ge \frac{1}{C_1}\mathrm{Tr}_{+,V}(q)\;,\;\;\text{for all}\;\; q\in\mathbb{R}^d\setminus \Sigma(\kappa) \;\text{with}\; \left|q\right|\ge C_1~,\label{1.4}
\end{align}
moreover if $\mathbb{R}^d\setminus \Sigma(\kappa)$ is not bounded
\begin{align}
\lim\limits_{\substack{q\to\infty \\ q\in \mathbb{R}^d\setminus \Sigma(\kappa)}}\frac{R_{V}^{^{\ge 3}}(q)^4}{\left|\mathrm{Hess}\;V(q)\right|}=0\;.\label{1.5}
\end{align}

\end{assumption}
Our main result is the following.
\begin{thm}\label{thm1.1}
Let $V(q)$ be a polynomial of degree $r$ greater than two verifying Assumption~\ref{assumption1}. Then
there exists a strictly positive constant $C_{V}>1$ (depending on $V$) such that
\begin{align}
\|K_{V}u\|^2_{L^2}+C_{V}\|u\|^2_{L^2}\ge \frac{1}{C_V}\Big(\|L(O_p)u\|^2_{L^2}&+\|L(\langle\nabla V(q)\rangle^{\frac{2}{3}})  u\|^2_{L^2}\nonumber\\&+\|L(\langle\mathrm{Hess}\; V(q)\rangle^{\frac{1}{2}} ) u\|^2_{L^2}+\|L(\langle D_q\rangle^{\frac{2}{3}} ) u\|^2_{L^2}\Big)~,\label{1.6}
\end{align}
holds for all $u\in D(K_{V})$ where $L(s)=\frac{s+1}{\log(s+1)}$ for any $s\ge1.$  
\end{thm}
\begin{cor}
\label{cor}
If $V(q)$ is  polynomial of degree greater than two that satisfies
Assumption~\ref{assumption1}, then the Kramers-Fokker-Planck operator $K_V$ has a compact resolvent.
\end{cor}
\begin{proof}{Proof of Corollary~\ref{cor}}\\
Assume $0<\delta<1.$ Define the functions $f_\delta:\mathbb{R}^d\to\mathbb{R}$ by $$f_\delta(q)= |\nabla V(q)|^{\frac{4}{3}(1-\delta)}+|\mathrm{Hess}\,V(q)|^{1-\delta}~.$$
From (\ref{1.6}) in Theorem~\ref{thm1.1} there is a constant $C_V>1$ such that 
\begin{align*}
\|K_Vu\|^2_{L^2}+C_V\|u\|^2_{L^2}\ge\frac{1}{C_V}\Big(\langle u,f_\delta u\rangle+ \|L(O_p)u\|^2_{L^2}+\|L(\langle D_q\rangle^{\frac{2}{3}} ) u\|^2_{L^2}\Big)~,
\end{align*}
holds for all $u\in\mathcal{C}_0^{\infty}(\mathbb{R}^{2d})$ and all $\delta\in (0,1).$
In order to prove that the operator $K_V$ has a compact resolvent it is sufficient to show that $\lim\limits_{q\to +\infty}f_\delta(q)=+\infty.$

To do so, assume $A>0$ and denote $\kappa=A^{\frac{1}{1-\delta}}.$ If $q\in\Sigma(\kappa),$ one has $$|\nabla V(q)|^{\frac{4}{3}(1-\delta)}\ge \kappa^{1-\delta}=A~.$$ Else if $q\in \mathbb{R}^d\setminus\Sigma(\kappa)$ by (\ref{1.5}) in Assumption~\ref{assumption1}, $\lim\limits_{\substack{q\to\infty \\ q\in \mathbb{R}^d\setminus \Sigma(\kappa)}} |\mathrm{Hess}\,V(q)|=+\infty.$ Hence there exists a constant $\eta>0$ such that  $|\mathrm{Hess}\,V(q)|^{1-\delta}\ge A$  for all $q\in\mathbb{R}^d\setminus \Sigma(\kappa)$ with $|q|\ge\eta.$

\end{proof}
\begin{RQ}
The results of Theorem \ref{thm1.1} and Corollary \ref{cor} can be extended in the case when $V=V_1+V_2$ where $V_1$ is polynomial satisfying Assumption \ref{assumption1} and $V_2$ is a function in $\mathcal{S}(\mathbb{R}^d).$ 
\end{RQ}
\section{Preliminary results}
This work is essentially based on two main strategies. The first one consists in the use of a partition of unity which is the most important tool that allows one to pass from
local to global estimates. 

In this paper, given a polynomial $V(q)$  we make use of a locally finite partition of unity with respect to the position variable $q\in\mathbb{R}^d$
 \begin{align}
\sum\limits_{j\in\mathbb{N}}\chi_j^2(q)=\sum\limits_{j\in\mathbb{N}}\widetilde{\chi}^2_j\Big(R_{V}^{^{\ge3}}(q_j)^{-1}(q-q_j)\Big)=1~\label{2.333M}
\end{align}
where $$\mathrm{supp}\;\widetilde{\chi}_j\subset B(q_j,a)\;\text{and }\;\widetilde{\chi}_j\equiv 1\;\text{in }\;\; B(q_j,b)$$ for some $q_j\in\mathbb{R}^d$ with $0<b<a$ independent of $j\in \mathbb{N}.$
Such a partition is described more precisely in Lemma~\ref{lemA.6} after taking $n=3.$ In our study introducing this partition yields to errors to be well controlled. 

The second approach lies in the decomposition of the operator $K_{V}$ onto two parts so that the first one be a Kramers-Fokker-Planck operator with polynomial potential of degree less than three. On this way, based on \cite{BNV}, we derive the result of Theorem~\ref{thm1.1}.

In order to prove Theorem~\ref{thm1.1} we need the following basic lemmas.
\begin{lem}\label{lem2.11}
Assume $V\in E_r$ with degree $r\in\mathbb{N}$. Consider the Kramers-Fokker-Planck operator $K_{V}$ defined as in (\ref{eq1}). For a locally finite partition of unity namely $\sum\limits_{j\in\mathbb{N}}\chi^2_j(q)=1$ one has 
\begin{align}
\|K_{V}u\|^2_{L^2(\mathbb{R}^{2d})}=\sum\limits_{j\in\mathbb{N}}\|K_{V}(\chi_ju)\|^2_{L^2(\mathbb{R}^{2d})}-\|(p\partial_q\chi_j)u\|^2_{L^2(\mathbb{R}^{2d})}\label{2.1}~,
\end{align}
for all $u\in\mathcal{C}_0^{\infty}(\mathbb{R}^{2d}).$

In particular when the degree of $V$ is larger than two and the cutoff functions $\chi_j$ have the form (\ref{2.333M}), there exists a constant $c_d>0$ (depending on the dimension $d$) so that 
\begin{align}
\|K_{V}u\|^2_{L^2(\mathbb{R}^{2d})}\ge\sum\limits_{j\in\mathbb{N}}\|K_{V}(\chi_ju)\|^2_{L^2(\mathbb{R}^{2d})}-c_dR_{V}^{^{\ge3}}(q_j)^2\|p\chi_ju\|^2_{L^2(\mathbb{R}^{2d})}\label{2.1}~,
\end{align}
holds for all $u\in\mathcal{C}_0^{\infty}(\mathbb{R}^{2d}).$
\end{lem} 
\begin{proof}
First let $V\in E_r$ with $r\in\mathbb{N}$ is the degree of $V.$ Assume that $u\in\mathcal{C}_0^{\infty}(\mathbb{R}^{2d}).$ On the one hand,
  \begin{align*} 
        \|K_{V}u\|^2_{L^{2}}=\sum\limits_{j\in\mathbb{N}}\langle K_{V}u,\chi_j^2K_{V}u\rangle=\sum\limits_{j\in\mathbb{N}}\langle u,K_{V}^*\chi_j^2K_{V}u\rangle\;.
  \end{align*}
 On the other hand,
  \begin{align*}
 \sum\limits_{j\in\mathbb{N}}\|K_{V}(\chi_j u)\|^2_{L^{2}}= \sum\limits_{j\in\mathbb{N}}\langle u,\chi_jK_{V}^*K_{V}\chi_ju\rangle\;.
  \end{align*}
Putting the above equalities together\begin{align*}
  \|K_{V}u\|^2_{L^{2}}-\sum\limits_{j\in\mathbb{N}}\|K_{V}(\chi_j u)\|^2_{L^{2}}=\sum\limits_{j\in\mathbb{N}}\langle u,(K_{V}^*\chi_j^2K_{V}-\chi_jK_{V}^*K_{V}\chi_j)u\rangle~.
  \end{align*}
 Using commutators, we compute \begin{align*}
  K_{V}^*\chi_j^2K_{V}&= K_{V}^*\chi_j[\,\chi_j,K_{V}\,]+K_{V}^*\chi_jK_{V}\chi_j\\&=K_{V}^*\chi_j[\,\chi_j,K_{V}\,]+[\, K_{V}^*,\chi_j\,]K_{V}\chi_j+\chi_jK_{V}^*K_{V}\chi_j\\&=K_{V}^*\chi_j[\,\chi_j,K_{V}\,]+[\, K_{V}^*,\chi_j\,]\Big([\, K_{V},\chi_j\,]+\chi_jK_{V}\Big)+\chi_jK_{V}^*K_{V}\chi_j~.
  \end{align*}
Thus
  \begin{align*}
  K_{V}^*\chi_j^2K_{V}-\chi_jK_{V}^*K_{V}\chi_j=K_{V}^*\chi_j[\,\chi_j,K_{V}\,]+[\, K_{V}^*,\chi_j\,]\chi_jK_{V}+[\, K_{V}^*,\chi_j\,]\circ[\, K_{V},\chi_j\,]~.
  \end{align*}
Now it is easy to check the following commutation relations  \begin{align*}
  \left\{
    \begin{array}{ll}
      \;[\,\chi_j,K_{V}\,]=-[\, K_{V},\chi_j\,]=-[\, p\partial_q,\chi_j(q)\,]=-p\partial_q\chi_j \\ \;
       [\, K_{V}^*,\chi_j\,]=[\,-p\partial_q,\chi_j(q)\,]=-p\partial_q\chi_j\\ 
       \; [\, K_{V}^*,\chi_j\,]\circ[\, K_{V},\chi_j\,]=-(p\partial_q\chi_j)^2~.      
    \end{array}
\right.
  \end{align*}
Collecting the terms, we
obtain \begin{align*}
 \sum\limits_{j\in\mathbb{N}}(K_{V}^*\chi_j^2K_{V}-\chi_jK_{V}^*K_{V}\chi_j)&=\sum\limits_{j\in\mathbb{N}}K_{V}^*\chi_j(-p\partial_q\chi_j)+(-p\partial_q\chi_j)\chi_jK_{V}-(p\partial_q\chi_j)^2\\&=\sum\limits_{j\in\mathbb{N}}K_{V}^*\Big(\partial_q(\frac{\chi_j^2}{2})\Big)-p\partial_q(\frac{\chi_j^2}{2})K_{V}-(p\partial_q\chi_j)^2&\\=-(p\partial_q\chi_j)^2~,
  \end{align*}
  where in the last line we make use simply $\sum\limits_{j\in\mathbb{N}}\chi^2_j(q)=1.$

From this follows immediately the identity
  \begin{align*}
  \|K_{V}u\|^2_{L^{2}}=\sum_{j\in\mathbb{N}}\Big(\|K_{V}(\chi_j u)\|^2_{L^{2}}-\|(p\partial_q\chi_j)u\|^2_{L^{2}}\Big)
  \end{align*}
 for all $u\in\mathcal{C}_0^{\infty}(\mathbb{R}^{2d}).$
 
Next, suppose that the degree of $V$ is greater than two and $\chi_j(q)=\widetilde{\chi}_j\Big(R_{{V}}^{^{\ge3}}(q_j)^{-1}(q-q_j)\Big)$  for all index $j$ and any $q\in\mathbb{R}^d$ with $$\mathrm{supp}\;\widetilde{\chi}_j\subset B(q_j,a)\;\text{and }\;\widetilde{\chi}_j\equiv 1\;\text{in }\; B(q_j,b)~.$$ Then we can write 
\begin{align*}
\sum\limits_{j\in\mathbb{N}}\|(p\partial_q\chi_j)u\|^2&=\sum\limits_{j\in\mathbb{N}}\sum\limits_{j'\in\mathbb{N}}\|(p\partial_q\chi_j)\chi_{j'}u\|^2\nonumber\\&\le c_d\sum\limits_{j\in\mathbb{N}}R_{V}^{^{\ge 3}}(q_j)^2\|p\chi_ju\|^2~,
\end{align*}
where $c_d$ is a constant that depends only on the dimension  $d.$ Here the last inequality is due to the fact that for each index $j$ there are finitely many $j'$ such that $(\partial_q\chi_j)\chi_{j'}$ is nonzero.
  \end{proof}
 Before stating the following lemma, we fix and remind some notations.
  \begin{NTS}
Let $V$ be a polynomial of degree r larger than two. Consider a locally finite partition of unity $\sum\limits_{j\in\mathbb{N}}\chi_j^2(q)=1$ described as in (\ref{2.333M}). 

Set for all $\kappa>0$
\begin{align*}
J(\kappa)=\Big\lbrace j\in \mathbb{N},\; \text{such that}\;\; \mathrm{supp}\;\chi_j\subset \Sigma(\kappa)\Big\rbrace~,
\end{align*}
where we recall that
\begin{align*}
\Sigma(\kappa)=\left\{q\in\mathbb{R}^d,\;\left|\nabla V(q)\right|^{\frac{4}{3}}\ge \kappa\Big(\left|\mathrm{Hess}\; V(q)\right|+R_{V}^{^{\ge 3}}(q)^4+1\Big)\right\}~.
\end{align*}
For a given $\kappa>0$ and all index $j\in \mathbb{N},$ let $V_{j}^2$ be the polynomial of degree less than three given by 
\begin{align}
V_{j}^2(q)=\sum_{0\le \left|\alpha\right|\le 2}\frac{\partial^{\alpha}_qV(q'_j)}{\alpha!}(q-q'_j)^{\alpha}~,
\end{align}
 where \begin{align*}
 \left\{
    \begin{array}{ll}
        q'_j=q_j & \mbox{if}\; j\in J(\kappa)\\
        q'_j\in  (\mathrm{supp}\;\chi_j)\cap\Big(\mathbb{R}^d\setminus \Sigma(\kappa)\Big) & \mbox{else.}
    \end{array}
\right.
\end{align*}
\end{NTS}
\begin{lem} \label{lem2.3}
Assume $V$ a polynomial of degree $r$ larger than two. Consider a locally finite partion of unity described as in (\ref{2.333M}). For a multi-index $\alpha\in \mathbb{N}^d
$ of length $\left|\alpha\right|\in \left\{1,2\right\}$ and all $j\in\mathbb{N},$ one has 
\begin{align}
\left|\partial_q^{\alpha}V(q)-\partial_q^{\alpha}V_{j}^2(q)\right| \le c_{\alpha,d,r} \Big(R_{V}^{^{\ge 3}}(q'_j)\Big)^{\left|\alpha\right|}\label{2.40}
\end{align}
for any $q\in \mathrm{supp}\;\chi_j=B(q_j,aR_{V}^{^{\ge3}}(q_j)^{-1}),$ where $c_{\alpha,d,r}=\sum\limits_{\substack{3\le\left|\beta\right|\le r }}\beta!\;a^{-\left|\beta\right|+\left|\alpha\right|}.$
  
As a consequence, if $V$ satisfies Assumption~\ref{assumption1}, 
 there exists a large constant $\kappa_1\ge\kappa_0$ so that for all $\kappa\ge\kappa_1$
 
$ \bullet\;\text{if}\;\; j\in J(\kappa)$
\begin{align}
2^{-1}\left|\partial_qV_{j}^2(q)\right|
\leq 
\left|\partial_qV(q)\right|\leq 2\left|\partial_qV_{j}^2(q)\right|
\label{B.20}
\quad \quad \text{for every}\;\;q\in \mathrm{supp}\;\chi_j\;,
\end{align}

$ \bullet\;\text{if}\;\; j\notin J(\kappa)$
\begin{align}
2^{-1}\left|\mathrm{Hess}\; V_{j}^{2}(q)\right|
\leq
\left|\mathrm{Hess}\; V(q)\right|\leq 2 \left|\mathrm{Hess}\; V_{j}^2(q)\right|\;,\label{B.21}
\end{align}
for any $q\in\mathrm{supp}\;\chi_j$ with $\left|q\right|\ge
C_2(\kappa)$ where $C_2(\kappa)>0$ is a large constant that depends on
$\kappa$. 
\end{lem}
\begin{proof}
Let $V$ be a polynomial of degree $r$ greater than two. In this proof we are going to need the following equivalence 
\begin{align}R_{V}^{^{\ge 3}}(q)\asymp R_{V}^{^{\ge 3}}(q')\;,
\end{align}
satisfied for all $q,q'\in\mathrm{supp}\;\chi_j$ and  proved in Lemma~\ref{A.5}. That is there is a constant $C>1$ such that for every $q,q'\in\mathrm{supp}\;\chi_j,$
\begin{align}
\Big(\frac{R_{V}^{^{\ge 3}}(q)}{R_{V}^{^{\ge 3}}(q')}\Big)^{\pm1}\le C~.\label{2.9m}
\end{align}
Assume $\alpha\in \mathbb{N}^d
$ of length $\left|\alpha\right|\in \left\{1,2\right\}.$ For every $j\in\mathbb{N}$, observe that 
\begin{align*}
\left|\partial_q^{\alpha}V(q)-\partial_q^{\alpha}V_{j}^2(q)\right|&=|\sum\limits_{\substack{3\le\left|\beta\right|\le r \\ \beta\ge \alpha}}\frac{\beta!}{(\beta-\alpha)!}\partial_q^{\beta}V(q'_j)(q-q'_j)^{\beta-\alpha}|\nonumber\\&\le \sum\limits_{\substack{3\le\left|\beta\right|\le r \\ \beta\ge \alpha}}\frac{\beta!}{(\beta-\alpha)!}\left|\partial_q^{\beta}V(q'_j)\right|\;\left|q-q'_j\right|^{\left|\beta\right|-\left|\alpha\right|}~,
\end{align*}
for any $q\in\mathbb{R}^d.$ Hence regarding the equivalence (\ref{2.9m}), there exists a constant $c_{\alpha,d,r}>0$ (depending as well on the multi-index $\alpha$, the dimesion $d$ and the degree r of $V$) so that 
 \begin{align}
\left|\partial_q^{\alpha}V(q)-\partial_q^{\alpha}V_{j}^2(q)\right|&\le \sum\limits_{\substack{3\le\left|\beta\right|\le r \\ \beta\ge \alpha}}\frac{\beta!}{(\beta-\alpha)!}\Big(R_{V}^{^{\ge 3}}(q'_j)\Big)^{\left|\beta\right|}\Big(aR_{V}^{^{\ge 3}}(q_j)\Big)^{-\left|\beta\right|+\left|\alpha\right|}\nonumber\\&\le\sum\limits_{\substack{3\le\left|\beta\right|\le r \\ \beta\ge \alpha}}\frac{\beta!}{(\beta-\alpha)!}a^{-\left|\beta\right|+\left|\alpha\right|}\Big(R_{V}^{^{\ge 3}}(q'_j)\Big)^{\left|\alpha\right|}\nonumber\\&\le c_{\alpha,d,r} \Big(R_{V}^{^{\ge 3}}(q'_j)\Big)^{\left|\alpha\right|}~,\label{2.77}
\end{align}
holds for all $q$ in the support of $\chi_j,$ where $c_{\alpha,d,r}=\sum\limits_{\substack{3\le\left|\beta\right|\le r }}\beta!\;a^{-\left|\beta\right|+\left|\alpha\right|}.$

In the rest of the proof, let the polynomial $V(q)$ satisfies Assumption~\ref{assumption1}. In vue of (\ref{2.77}), we get when $\left|\alpha\right|=1$ 
\begin{align}
\left|\nabla V(q)-\nabla V_{j}^2(q)\right|&\le c_{1,d,r}\;R_{V}^{^{\ge 3}}(q'_j)~,\label{A.23MM}
\end{align}
for all $j\in \mathbb{N}$ and any $q\in \mathrm{supp}\;\chi_j,$  where $c_{1,d,r}=\sum\limits_{\substack{3\le\left|\beta\right|\le r }}\beta!\;a^{-\left|\beta\right|+1}.$
 Given $\kappa\ge\kappa_0,$
assume first that $ j\in J(\kappa).$
 By virtue of the equivalence (\ref{2.9m}), it results from (\ref{A.23MM})
\begin{align}
\left|\nabla V(q)-\nabla V_{j}^2(q)\right|\le c_{1,d,r}C\;R_{V}^{^{\ge 3}}(q)\label{A.23}\;,
\end{align}
for every $q\in\mathrm{supp}\;\chi_j.$
Then we obtain \begin{align}
\left|\nabla V(q)-\nabla V_{j}^2(q)\right|&\le\frac{c_{1,d,r}C}{\kappa^{\frac{1}{4}}}\left|\nabla V(q)\right|^{\frac{1}{3}}\nonumber\\&\le \frac{c_{1,d,r}C}{\kappa^{\frac{1}{4}}}\left|\nabla V(q)\right|\label{A.24}
\end{align}
for all $q\in\mathrm{supp}\;\chi_j.$ For the above second inequality we know that $|\nabla V(q)|\ge 1$ for every $q\in\mathrm{supp}\;\chi_j$, indeed since $j\in J(\kappa),$  \begin{align*}|\nabla V(q)|\ge \kappa^{\frac{3}{4}}\ge\kappa_0^{\frac{3}{4}}\ge1~.\end{align*} 
Taking the constant $\kappa_1\ge\kappa_0$ such that $\frac{c_{1,d,r}C}{\kappa_1^{\frac{1}{4}}}\le \frac{1}{2},$ we get for every $\kappa\ge\kappa_1$ 
\begin{align*}
 \Big|\left|\nabla V(q)\right|-|\nabla V_{j}^2(q)|\Big|\le |\nabla V(q)-\nabla V_{j}^2(q)|\le \frac{1}{2}|\nabla V(q)|~,
\end{align*}
for any $q\in \mathrm{supp}\;\chi_j$ when $j\in J(\kappa).$
Therefore
\begin{align*}
\frac{1}{2}|\nabla V_{j}^2(q)|\le\left|\nabla V(q)\right|\le \frac{3}{2}|\nabla V_{j}^2(q)|
\end{align*}
holds for all $q\in \mathrm{supp}\;\chi_j$ when $j\in J(\kappa).$

On the other hand when $|\alpha|=2,$ by (\ref{2.77}) and (\ref{2.9m}) there is a constant $c_{2,d,r}>0$ so that for all $j\in\mathbb{N}$
\begin{align}
|\partial_q^{\alpha}V(q)-\partial_q^{\alpha}V_{j}^2(q)|&\le  c_{2,d,r}C^2R_{V}^{^{\ge 3}}(q)^{2}\;.\label{2.1000}
\end{align}
holds for every $q\in \mathrm{supp}\;\chi_j,$ where $c_{2,d,r}=\sum\limits_{\substack{3\le\left|\beta\right|\le r }}\beta!\;a^{-\left|\beta\right|+2}.$ 
Given $\kappa\ge \kappa_0$ assume now $ j\not\in J(\kappa).$
Using the fact that $R_V^{^{\ge3}}(q)\ge R_V^{=r}(0)$ for every $q\in\mathbb{R}^d,$  we derive from (\ref{2.1000}) that 
\begin{align*}
|\partial_q^{\alpha}V(q)-\partial_q^{\alpha}V_{j}^2(q)|\le  c_{2,d,r}C^2\frac{R_{V}^{^{\ge 3}}(q)^{4}}{R_V^{=r}(0)^2}~,
\end{align*}
for all $q\in \mathrm{supp}\;\chi_j.$

Assuming $\kappa\ge\kappa_0$ and $\;j\notin J(\kappa),$ we obtain using the previous inequality and applying Lemma~\ref{lem2.1}
\begin{align*}
\Big|\sum\limits_{|\alpha|=2}|\partial_q^{\alpha}V(q)|-\sum\limits_{|\alpha|=2}|\partial_q^{\alpha}V_{j}^2(q)|\Big|\le \sum\limits_{|\alpha|=2}|\partial_q^{\alpha}V(q)-\partial_q^{\alpha}V_{j}^2(q)|\le\frac{1}{2}|\mathrm{Hess}\; V(q)|\;,
\end{align*}
for any $q\in\mathrm{supp}\;\chi_j$ with $|q|\ge C_2(\kappa)$ where $C_2(\kappa)$ is a strictly positive large constant depending on $\kappa$\,. In other words,
\begin{align*}
\frac{1}{2}|\mathrm{Hess}\; V_{j}^2(q)|\le|\mathrm{Hess}\; V(q)|\le \frac{3}{2}|\mathrm{Hess}\; V_{j}^2(q)|
\end{align*}
holds for all  $q\in \mathrm{supp}\;\chi_j$ with $|q|\ge C_2(\kappa)$ and $j\notin J(\kappa).$
\end{proof}
\begin{lem}\label{lem2.3M}
Given two positive operators $A$ and $B$  such that $$\|u\|^2< \langle u,Au\rangle\le \langle u,Bu\rangle$$ for all $u\in \mathcal{D}$ where $\mathcal{D}$ is dense in $D(A^{1/2}),$ one has
\begin{align}\langle u,\frac{A^{\alpha_0}}{(\log(A^{\alpha_0/2}))^k}u\rangle\le \langle u,\frac{B^{\alpha_0}}{(\log(B^{\alpha_0/2}))^k}u\rangle~,\label{2.144M}\end{align}
for all $u\in \mathcal{D},$ any $\alpha_0\in[0,1]$ and every natural number $k.$
\end{lem}
\begin{proof}
Assume that $A,B$ are two positive operators so that  \begin{align}
\|u\|^2< \langle u,Au\rangle\le \langle u,Bu\rangle~,\label{2.155MM}\end{align}
holds for all $u\in\mathcal{D}.$
Referring to \cite{Sim} (see Proposition 6.7 and Example 6.8), for any positive operator $C$ and every $\alpha\in(0,1)$ we can write 
\begin{align}C^{\alpha}=\frac{2\sin(\pi\alpha)}{\pi}\int_0^{+\infty}w^{\alpha-1}(C+w)^{-1}Cdw~.\label{2.2215M}\end{align}
From (\ref{2.155MM}) and (\ref{2.2215M})
 \begin{align}
\|u\|^2< \langle u,A^{\alpha}u\rangle\le \langle u,B^{\alpha}u\rangle~,\label{2.115M}
\end{align} for any  $u\in \mathcal{D}$ and every $\alpha\in[0,1].$

Furthermore, for any positive operator $C$ with domain $D(C)$ we define its logarithm for all $u\in D(C)$ by 
\begin{align}
\langle u,\log(C)u\rangle=\lim\limits_{\alpha \to0^+} \langle u,\frac{C^{\alpha}-1}{\alpha}u\rangle~,\label{2.18MM}
\end{align}
where the operator $C^{\alpha}$ is given in (\ref{2.2215M}).

Using (\ref{2.155MM}) and (\ref{2.18MM}) \begin{align}\|u\|^2< \langle u,\log(A)u\rangle\le \langle u,\log(B)u\rangle~,\label{2.116M}
\end{align}
holds for all $u\in \mathcal{D}.$ Integrating (\ref{2.115M}) with respect to $\alpha$ over $[0,\alpha_0]$ where $\alpha_0\in[0,1]$ we get 
\begin{align}\langle u,\frac{1}{\log(A)}(A^{\alpha_0}-I)u\rangle\le \langle u,\frac{1}{\log(B)}(B^{\alpha_0}-I)u\rangle~.\label{2.161M}
\end{align}
Furthermore by (\ref{2.116M}) \begin{align}
\langle u,\frac{1}{\log(B)}u\rangle\le \langle u,\frac{1}{\log(A)}u\rangle<\|u\|^2~.\label{2.171M}
\end{align}
Therefore from (\ref{2.161M}) and (\ref{2.171M})
$$\langle u,\frac{A^{\alpha_0}}{\log(A)}u\rangle\le \langle u,\frac{B^{\alpha_0}}{\log(B)}u\rangle~.$$
holds for any $\alpha_0\in[0,1].$
Then by induction on $k\in\mathbb{N},$ we obtain 
$$\langle u,\frac{A^{\alpha_0}}{(\log(A))^k}u\rangle\le \langle u,\frac{B^{\alpha_0}}{(\log(B))^k}u\rangle~,$$
for all $\alpha_0\in[0,1]$ and every natural number $k.$
Or equivalently 
$$\langle u,\frac{A^{\alpha_0}}{(\log(A^{\alpha_0/2}))^k}u\rangle\le \langle u,\frac{B^{\alpha_0}}{(\log(B^{\alpha_0/2}))^k}u\rangle~,$$
for every $\alpha_0\in[0,1],\;k\in\mathbb{N}.$

\end{proof}
\begin{lem}\label{lem2.2} Assume $V(q)$ a polynomial of degree greater than two. Let $\sum\limits_{j\in\mathbb{N}}\chi^2_j(q)$ be a locally finite  partition of unity defined as in (\ref{2.333M}).

There is a constant $c>0$ such that
\begin{align}
\langle u,(1+D_q^2+R^{^{\ge3}}_{V}(q)^4)^{\alpha}u\rangle\le c\sum\limits_{j\in\mathbb{N}}\langle u,\chi_j(1+D_q^2+R^{^{\ge3}}_{V}(q'_j)^4)^{\alpha}\chi_ju\rangle~,
\end{align}
is valid for all $u\in\mathcal{C}_0^{\infty}(\mathbb{R}^{2d})$ and any $\alpha\in[0,1].$

As a consequence, there exists a constant $c>0$ so that
\begin{align}\sum_{j\in\mathbb{N}}\|L\Big((1+D_q^2+R_{V}^{^{\ge3}}(q'_j)^4)^{\frac{1}{3}}\Big)\chi_ju\|^2\ge \frac{1}{c}\|L\Big((1+D_q^2+R_{V}^{^{\ge3}}(q)^4)^{\frac{1}{3}}\Big)u\|^2~,\label{2.17M}\end{align}
holds for all $u\in\mathcal{C}_0^{\infty}(\mathbb{R}^{2d}),$ where $L(s)=\frac{s+1}{\log(s+1)}$ for all $s\ge1.$
\end{lem}
\begin{proof} We first set $E_0=L^2(\mathbb{R}^{2d})$ and $E_1=\left\{u\in L^2(\mathbb{R}^{2d}),\;\langle u,(1-\Delta_q+R_{V}^{^{\ge 3}}(q)^4)u\rangle<+\infty\right\}$ endowed respectively with the norms $\|\cdot\|_{E_0}=\|\cdot\|_{L^2(\mathbb{R}^{2d})}$ and $\|\cdot\|_{E_1}$ defined as follows for all $u\in L^2(\mathbb{R}^{2d})$ \begin{align*}
\|u\|^2_{E_1}&=\|u\|^2_{L^2(\mathbb{R}^{2d})}+\|D_qu\|^2_{L^2(\mathbb{R}^{2d})}+\|R_{V}^{^{\ge3}}(q)^2u\|^2_{L^2(\mathbb{R}^{2d})}\\&=\|(1-\Delta_q+R_{V}^{^{\ge3}}(q)^4)^{1/2}u\|^2_{L^2(\mathbb{R}^{2d})}~.
\end{align*}
By Simader theorem (which states that if $W\in\mathcal{C}^{\infty}(\mathbb{R}^d)$ and $-\Delta+W(x)$ is a symetric non negative operator on $\mathcal{C}_0^{\infty}(\mathbb{R}^{d})$ then $-\Delta+W(x)$ is essentially self adjoint on $\mathcal{C}_0^{\infty}(\mathbb{R}^{d})$), the operator $1-\Delta_q+R_{V}^{^{\ge3}}(q)^4$ is essentially self adjoint on $\mathcal{C}_0^{\infty}(\mathbb{R}^{2d})$ and hence $E_1$ corresponds to the spectrally defined subspace of $L^2(\mathbb{R}^{2d})$.

Given a partition of unity as in (\ref{2.333M}), define the linear map \begin{align*}T:E_0\to (L^2(\mathbb{R}^{2d}))^{\mathbb{N}},\;u\mapsto (u_j)_{j\in\mathbb{N}}=(\chi_ju)_{j\in\mathbb{N}}~,
\end{align*} and denote $F_0:=\mathrm{Im}\;T.$
Notice that $T:E_0\to F_0$ is a unitary isometry. Indeed for all $u\in E_0,$ \begin{align}\;\|Tu\|^2_{F_0}=\sum\limits_{j\in\mathbb{N}} \|\chi_ju\|^2_{L^2}=\|u\|^2_{L^2}=\|u\|^2_{E_0}~,\label{2.14M}\end{align} further the inverse map of T is well defined by \begin{align*}T^{-1}:F_0\to E_0,\;(u_j)_{j\in\mathbb{N}}\mapsto u=\sum\limits_{j\in\mathbb{N}}\chi_ju_j~.\end{align*}
Now introduce the set 
\begin{align*}
F_1=\left\{(u_j)_{j\in\mathbb{N}}\in F_0,\;\sum\limits_{j\in\mathbb{N}}\langle u_j,(1-\Delta_q+R_{V}^{^{\ge3}}(q_j')^4)u_j\rangle<+\infty\right\}~,
\end{align*}
with its associated norm defined for all $(u_j)_{j\in\mathbb{N}}\in F_1$ by  
\begin{align*}
\|(u_j)_{j\in\mathbb{N}}\|^2_{F_1}&=\sum\limits_{j\in\mathbb{N}}\Big(\|u_j\|^2_{L^2(\mathbb{R}^{2d})}+\|D_qu_j\|^2_{L^2(\mathbb{R}^{2d})}+\|R_{V}^{^{\ge3}}(q_j')^2u_j\|^2_{L^2(\mathbb{R}^{2d})}\Big)\\&=\sum\limits_{j\in\mathbb{N}}\|(1-\Delta_q+R_{V}^{^{\ge3}}(q_j')^4)^{1/2}u_j\|^2_{L^2(\mathbb{R}^{2d})}~.
\end{align*}
Assume $u\in E_0.$ For all $j\in\mathbb{N},$ let $q_j'\in \mathrm{supp}\;\chi_j.$ Observe that
\begin{align}
|\;\|Tu\|^2_{F_1}-\|u\|^2_{E_1}|&=|\sum\limits_{j\in\mathbb{N}}\langle u_j,(1-\Delta_q+R_{V}^{^{\ge3}}(q_j')^4)u_j\rangle-\langle u,(1-\Delta_q+R_{V}^{^{\ge3}}(q)^4)u\rangle|\nonumber\\&=|\sum\limits_{j\in\mathbb{N}}\langle u_j,-\Delta_qu_j\rangle-\langle u,-\Delta_qu\rangle+\sum\limits_{j\in\mathbb{N}}\langle u_j,(R_{V}^{^{\ge3}}(q_j')^4-R_{V}^{^{\ge3}}(q)^4)u_j\rangle|\nonumber\\&\le |\sum\limits_{j\in\mathbb{N}}\langle u_j,-\Delta_qu_j\rangle-\langle u,-\Delta_qu\rangle|+\sum\limits_{j\in\mathbb{N}}\langle u_j,|R_{V}^{^{\ge3}}(q_j')^4-R_{V}^{^{\ge3}}(q)^4|u_j\rangle~.\label{2.233M}
\end{align}
Since we are dealing with cutoff functions satisfying $\sum\limits_{j\in\mathbb{N}}|\nabla\chi_j|^2\le c\,R_{V}^{^{\ge3}}(q)^2$ and owning to the equivalence $R_{V}^{^{\ge3}}(q)\asymp R_{V}^{^{\ge3}}(q'_j)$ for all $q\in\mathrm{supp}\;\chi_j$, it follows from (\ref{2.233M})
\begin{align*}
|\;\|Tu\|^2_{F_1}-\|u\|^2_{E_1}|\le c_1\sum\limits_{j\in\mathbb{N}}\langle u_j,R_{V}^{^{\ge3}}(q_j')^4u_j\rangle\le c_1\|Tu\|^2_{F_1}~,
\end{align*}
and \begin{align*}
|\;\|Tu\|^2_{F_1}-\|u\|^2_{E_1}|\le c'_1\langle u,R_{V}^{^{\ge3}}(q)^4u\rangle\le c_1'\|u\|^2_{E_1}~,
\end{align*}
where $c_1,c_1'$ are two strictly positive constants.
As a result \begin{align}\frac{1}{\sqrt{(c_1+1)}}\|u\|_{E_1}\le\|Tu\|_{F_1}\le \sqrt{(c_1'+1)}\|u\|_{E_1}~.\label{2.15M}\end{align}
In view of (\ref{2.14M}) and (\ref{2.15M}), we conclude by interpolation that for all $\alpha\in[0,1]$
\begin{align*}
T:E_{\alpha}\to F_{\alpha},
\end{align*}
verifies $\|T\|_{\mathcal{L}(E_{\alpha},F_{\alpha})}\le c^{\alpha}$ and $\|T^{-1}\|_{\mathcal{L}(F_{\alpha},E_{\alpha})}\le c^{\alpha}$, where $E_{\alpha}$ and $F_{\alpha}$ are two interpolated spaces endowed respectively with the norms 
\begin{align*}
\|u\|_{E_{\alpha}}=\|(1-\Delta_q+R_{V}^{^{\ge3}}(q)^4)^{\alpha/2}u\|_{L^2(\mathbb{R}^{2d})}~,
\end{align*}
and 
\begin{align*}
\|(v_j)_{j\in \mathbb{N}}\|_{F_{\alpha}}=\sum\limits_{j\in\mathbb{N}}\|(1-\Delta_q+R_{V}^{^{\ge3}}(q_j')^4)^{\alpha/2}u_j\|_{L^2(\mathbb{R}^{2d})}~.
\end{align*}
Hence there is a constant $c>0$ so that 
\begin{align}
\langle u,(1+D_q^2+R^{^{\ge3}}_{V}(q)^4)^{\alpha}u\rangle\le c\sum\limits_{j\in\mathbb{N}}\langle u,\chi_j(1+D_q^2+R^{^{\ge3}}_{V}(q'_j)^4)^{\alpha}\chi_ju\rangle~,\label{2.99}
\end{align}
holds for all $u\in\mathcal{C}_0^{\infty}(\mathbb{R}^{2d})$ and any $\alpha\in[0,1].$
In order to prove (\ref{2.17M}), repeat the same process as in Lemma~\ref{lem2.3M}. Starting with
\begin{align}
\|u\|^2< \langle u,(1+D_q^2+R^{^{\ge3}}_{V}(q)^4)^{\alpha}u\rangle\le c\sum\limits_{j\in\mathbb{N}}\langle u,\chi_j(1+D_q^2+R^{^{\ge3}}_{V}(q'_j)^4)^{\alpha}\chi_ju\rangle~,\label{2.224M}
\end{align}
for all $u\in\mathcal{C}_0^{\infty}(\mathbb{R}^{2d})$ and any $\alpha\in[0,1],$ remark that when integrating over $\alpha\in[0,\frac{2}{3}]$ we can interchange the sum and the integral in the left hand side of (\ref{2.224M}) since the partition of unity is locally finite.

This completes the proof.
\end{proof}
\section{Proof of Theorem~\ref{thm1.1}}
In this section we present the proof of Theorem~\ref{thm1.1}. In the
sequel for a given polynomial $V(q)$ with degree r greater than two, we always use a locally finite partition of unity \begin{align*}\sum\limits_{j\in\mathbb{N}}\chi_j^2(q)=\sum\limits_{j\in\mathbb{N}}\widetilde{\chi}_j^2\Big(R_{V}^{^{\ge3}}(q_j)^{-1}(q-q_j)\Big)=1~,\end{align*}where $$\mathrm{supp}\;\widetilde{\chi}_j\subset B(q_j,a)\;\text{and }\;\widetilde{\chi}_j\equiv 1\;\text{in }\; B(q_j,b)$$ for some $q_j\in\mathbb{R}^d$ with $0<b<a$ independent of the natural numbers $j,$ defined more specifically as in Lemma~\ref{lemA.6} with $n=3.$
\begin{proof}
Let $V(q)$ be a polynomial with degree larger than two that satisfies Assumption~\ref{assumption1}. Assume $u\in\mathcal{C}_0^{\infty}(\mathbb{R}^{2d}).$ In the whole proof we denote $u_j=\chi_ju$ for all natural number $j.$

From Lemma~\ref{lem2.11} we get
\begin{align}
\|K_{V}u\|^2_{L^2(\mathbb{R}^{2d})}\ge\sum\limits_{j\in\mathbb{N}}\|K_{V}u_j\|^2_{L^2(\mathbb{R}^{2d})}-c_dR_{V}^{^{\ge3}}(q_j)^2\|pu_j\|^2_{L^2(\mathbb{R}^{2d})}\label{2.1}~.
\end{align} 
Given $\kappa\ge\kappa_0,$ set
\begin{align*}
J(\kappa)=\Big\lbrace j\in \mathbb{N},\; \text{such that}\;\; \mathrm{supp}\;\chi_j\subset \Sigma(\kappa)\Big\rbrace~.\end{align*}
For all index $j\in \mathbb{N},$ let $V_{j}^2$ be the polynomial of degree less than three given by 
\begin{align*}
V_{j}^2(q)=\sum_{0\le \left|\alpha\right|\le 2}\frac{\partial^{\alpha}_qV(q'_j)}{\alpha!}(q-q'_j)^{\alpha}~,
\end{align*}
 where \begin{align*}\left\{
    \begin{array}{ll}
        q'_j=q_j & \mbox{if}\; j\in J(\kappa)\\
        q'_j\in  (\mathrm{supp}\;\chi_j)\cap\Big(\mathbb{R}^d\setminus \Sigma(\kappa)\Big) & \mbox{else.}
    \end{array}
\right.\end{align*}
  
We associate with each polynomial $V_{j}^2$ the Kramers-Fokker-Planck operator $K_{V_{j}^2}.$ Observe that using the parallelogram law $2(\|x\|^2+\|y\|^2)-\|x+y\|^2=\|x-y\|^2\geq 0,$
 \begin{align}
\sum\limits_{j\in \mathbb{N}}\|K_{V}u_j\|^2_{L^2(\mathbb{R}^{2d})}&=\sum\limits_{j\in \mathbb{N}}\|K_{V_{j}^2}u_j+(K_{V}-K_{V_{j}^2})u_j\|^2_{L^2(\mathbb{R}^{2d})}\nonumber\\&\ge \frac{1}{2}\sum\limits_{j\in \mathbb{N}}\|K_{V_{j}^2}u_j\|^2_{L^2(\mathbb{R}^{2d})}-\|(\nabla V(q)-\nabla V_{j}^2(q))\partial_pu_j\|^2_{L^2(\mathbb{R}^{2d})}\label{2.3}
\end{align}
On the other hand, by (\ref{2.40}) in Lemma~\ref{lem2.3} 
\begin{align}
\sum\limits_{j\in \mathbb{N}}\|(\nabla V(q)-\nabla V_{j}^2(q))\partial_pu_j\|^2_{L^2(\mathbb{R}^{2d})}\le c_{1,d,r}\sum\limits_{j\in \mathbb{N}}R_{V}^{^{\ge 3}}(q_j')^2\|\partial_pu_j\|^2_{L^2(\mathbb{R}^{2d})}~.\label{2.4}
\end{align}
Combining (\ref{2.1}), (\ref{2.3}) and (\ref{2.4}) we get immediately 
\begin{align*}
\|K_{V}u\|^2_{L^2(\mathbb{R}^{2d})}&\ge\frac{1}{2}\sum\limits_{j\in \mathbb{N}}\|K_{V_{j}^2}u_j\|^2_{L^2(\mathbb{R}^{2d})}-c_{1,d,r}R_{V}^{^{\ge 3}}(q'_j)^2\|\partial_pu_j\|^2_{L^2(\mathbb{R}^{2d})}-c_dR_{V}^{^{\ge3}}(q_j)^2\|pu_j\|^2_{L^2(\mathbb{R}^{2d})}
\end{align*}
Therefore, making use of the equivalence $(\ref{A.55}),$ it follows
\begin{align}
\|K_{V}u\|^2_{L^2(\mathbb{R}^{2d})}&\ge\frac{1}{2}\sum\limits_{j\in \mathbb{N}}\|K_{V_{j}^2}u_j\|^2_{L^2(\mathbb{R}^{2d})}-c'_{d,r}R_{V}^{^{\ge 3}}(q'_j)^2\langle u_j,O_pu_j\rangle_{L^2(\mathbb{R}^{2d})}\;,\label{3.555}
\end{align}
where $c'_{d,r}=2(c_{1,d,r}^2+c_d).$

Using respectively the Cauchy Schwarz inequality then the Cauchy inequality with epsilon (for any real numbers $a,b$ and all $\epsilon>0,\;ab\le \epsilon a^2+\frac{1}{4\epsilon}b^2$),
\begin{align*}
c'_{d,r}R_{V}^{^{\ge 3}}(q'_j)^2\langle u_j,O_pu_j\rangle&=c'_{d,r}R_{V}^{^{\ge 3}}(q'_j)^2\mathrm{Re}\langle u_j,K_{V_{j}^2}u_j\rangle\\&\le c'_{d,r}R_{V}^{^{\ge 3}}(q'_j)^2\| u_j\|\cdot\|K_{V_{j}^2}u_j\|\\&\le \Big(c'_{d,r} R_{V}^{^{\ge 3}}(q'_j)^2\Big)^2\|u_j\|^2+\frac{1}{4}\|K_{V_{j}^2}u_j\|^2~.
\end{align*}
Putting the above estimate and (\ref{3.555}) together we obtain
\begin{align}
\|K_{V}u\|^2\ge \sum\limits_{j\in \mathbb{N}}\frac{1}{4}\|K_{V_{j}^2}u_j\|^2-(c'_{d,r})^2R_{V}^{^{\ge 3}}(q'_j)^4\|u_j\|^2~.\label{3.666M}
\end{align}
From now on assume $\kappa\ge\kappa_1,$ where $\kappa_1\ge \kappa_0$
is introduced in Lemma~\ref{lem2.3}.
Remember as well that 
the constants $C_1,C_2(\kappa)$ are given respectively in
Assumption~\ref{assumption1} (see (\ref{1.4})) and Lemma~\ref{lem2.3}
(see \eqref{B.21}).  By introducing $C(\kappa)\geq \max
(C_{1},C_{2}(\kappa))$\,, which will be fixed later, we  set
for each $\kappa,$ 
\begin{align*}I(\kappa)=\Big\lbrace j\in \mathbb{N},\; \text{such that}\;\; \mathrm{supp}\;\chi_j\subset\left\{q\in\mathbb{R}^d,\;\;|q|\ge C(\kappa)\right\}\Big\rbrace~.\end{align*}
The rest of the proof is divided into three steps. The first one is
devoted to the control of the terms in the the left hand side of
(\ref{3.666M}) for which $j\in I(\kappa)$ for some large $\kappa\ge
\kappa_0$ to be chosen. At the end of the first step the constants
$\kappa>\kappa_{1}$ and $C(\kappa)\geq \max (C_{1},C_{2}(\kappa))$
will be fixed.
 So on, the second step is concerned with
 the remaining terms for which the support of the cutoff functions
 $\chi_j$ are included in some closed ball $B(0,C'(\kappa))$. We
 finally  sum up all the terms and refer to Lemma~\ref{lem2.2}
 after some elementary optimization trick in the last step.
\smallskip

\noindent\textbf{Step 1, $\mathbf{j\in I(\kappa)}$, $\mathbf{\kappa\geq
  \kappa_{1}}$ to be fixed:} 
As proved in \cite{BNV}, there is a constant $c>0$ such that for all $j\in I(\kappa)$
\begin{align}
\|K_{V_{j}^2}u_j\|^2+A_{V_{j}^2}\|u_j\|^2\ge c\Big(\|O_pu_j\|^2+\|\langle\partial_q V_{j}^2(q)\rangle^{2/3}u_j\|^2+\|\langle D_q\rangle^{2/3}u_j\|\Big)~,\label{3.6M}
\end{align}
where \begin{align*}
 A_{V_{j}^2} &=\max\{(1+\mathrm{Tr}_{+,V_{j}^2})^{2/3}, 1+\mathrm{Tr}_{-,V_{j}^2}\}\\&=\max\{(1+\mathrm{Tr}_{+,V}(q'_j))^{2/3}, 1+\mathrm{Tr}_{-,V}(q'_j)\}~. 
\end{align*}
Hence there is a constant $C>0$ so that 
\begin{align}
\|K_{V_{j}^2}u_j\|^2+(1+10C)t_j^4\|u_j\|^2\ge C\Big(\|O_pu_j\|^2&+\|\langle\partial_q V_{j}^2(q)\rangle^{2/3}u_j\|^2\nonumber\\&+\|\langle D_q\rangle^{2/3}u_j\|+10Ct_j^4\|u_j\|^2\Big)\label{2.6}\;,
\end{align}
where we use the notation $t_j=2\langle \mathrm{Hess}\;V(q'_j)\rangle^{1/4}$ in the whole of the proof.

Recall that as mentioned in \cite{BNV}, the constant $c$ in (\ref{3.6M}) does not depend on the polynomial $V_{j}^2$ and then so is the constant $C$ in (\ref{2.6}).
\newline
Now for all index $j\in I(\kappa)$ we distinguish two cases: either $j\in J(\kappa)$ or $j\not\in J(\kappa).$
\smallskip

\noindent\textbf{Case 1.}  Assume $j\in J(\kappa)$. Then taking into account the inequality (\ref{B.20}) in Lemma~\ref{lem2.3} and using the estimate (\ref{2.6}) we obtain
\begin{align}
\|K_{V_j^2}u_j\|^2+(1+10C)t_j^4\|u_j\|^2\ge C\Big(\|O_pu_j\|^2&+\|\langle\partial_q V(q)\rangle^{2/3}u_j\|^2\nonumber\\&+\|\langle  D_q\rangle^{2/3}u_j\|+10t_j^4\|u_j\|^2\Big)~,\label{3.111111}
\end{align}
Furthermore, since for all index $j\in\mathbb{N}$ the quantity  $R_{V}^{^{\ge2}}(q'_j)^2$ is always greater than $|\mathrm{Hess}\;V(q_j')|,$ there exists a constant $c_{d}>0$ so that for every $j\in J(\kappa),$
\begin{align*}
t_j^4=16\langle\mathrm{Hess}\;V(q_j')\rangle\le c_{d}\langle R_{V}^{^{\ge2}}(q'_j)^2\rangle~.
\end{align*}
Using the fact that the metric $R_{V}^{^{\ge2}}(q)\,dq^2$ is $R_{V}^{^{\ge3}}(q)\,dq^2$ slow (see Definition (\ref{def}) and Lemma\;\ref{A.5}), it follows 
\begin{align*}
t_j^4\le c_{d}\langle R_{V}^{^{\ge2}}(q)^2\rangle\;,
\end{align*}
for every $q\in \mathrm{supp}\; \chi_j.$
Hence there is a constant $c_d'>0$ (depending on the dimension $d$) such that
\begin{align*}
t_j^4&\le c_{d}\langle\Big(\sum\limits_{|\alpha|=2}|\partial_q^{\alpha}V(q)|^{\frac{1}{|\alpha|}}+R_{V}^{^{\ge3}}(q)\Big)^2\rangle\nonumber\\&\le 3c_{d}\langle(\sum\limits_{|\alpha|=2}|\partial_q^{\alpha}V(q)|^{\frac{1}{|\alpha|}})^2+R_{V}^{^{\ge3}}(q)^2\rangle\\&\le c'_d\langle|\mathrm{Hess}\;V(q)|+R_{V}^{^{\ge3}}(q)^2\rangle~,
\end{align*}
holds for any $q\in \mathrm{supp}\; \chi_j.$
Or since for every $q\in\mathbb{R}^d$ on has $R_{V}^{^{\ge3}}(q)\ge R_{V}^{=r}(0)~,$ we derive from the previous estimate that for any $q\in\mathrm{supp}\;\chi_j,$
\begin{align}
t_j^4&\le c'_d\langle |\mathrm{Hess}\;V(q)|+\frac{R_{V}^{^{\ge3}}(q)^4}{R_{V}^{=r}(0)^2}\rangle\nonumber\\&\le \frac{c"_d}{\kappa}\max(1,R_V^{=r}(0)^{-2})\langle\partial_qV(q)\rangle^{\frac{4}{3}}~.\label{3.9M}
\end{align}
Collecting the estimates (\ref{3.111111}) and (\ref{3.9M}), we get for  $\kappa\ge \kappa_1$  
\begin{align*}
\|K_{V_{j}^2}u_j\|^2+(1+10C) \frac{c"_d}{\kappa}&\max(1,R_V^{=r}(0)^{-2})\|\langle\partial_qV(q)\rangle^{\frac{2}{3}}u_j\|^2\\&\ge C\Big(\|O_pu_j\|^2+\|\langle \partial_q V(q)\rangle^{2/3}u_j\|^2+\|\langle D_q\rangle^{2/3}u_j\|^2+10t_j^4\|u_j\|^2\Big)~.
\end{align*}
Choosing $\kappa_2\ge\kappa_1$ so that $$\frac{C}{2}\ge(1+10C) \frac{c"_d}{\kappa_2}\max(1,R_V^{=r}(0)^{-2})~,$$
the following inequality
\begin{align}
\|K_{V_{j}^2}u_j\|^2\ge C\Big(\|O_pu_j\|^2+\frac{1}{2}\|\langle \partial_q V(q)\rangle^{2/3}u_j\|^2+\|\langle D_q\rangle^{2/3}u_j\|^2+10t_j^4\|u_j\|^2\Big)~,\label{3.13M}
\end{align}
holds for all $j\in J(\kappa)$ with $\kappa\ge \kappa_2.$

Or since $j\in J(\kappa),$ there is a constant $c_1>0$ so that 
\begin{align}
\frac{1}{8}\langle\partial_qV(q)\rangle^{\frac{4}{3}}\ge c_1\langle\mathrm{Hess}\;V(q)\rangle~,\label{3.111M}
\end{align}
holds for all $q\in\mathrm{supp}\;\chi_j.$ In addition, using the equivalence (\ref{A.55}) it follows
\begin{align}
\frac{1}{8}\langle\partial_qV(q)\rangle^{\frac{4}{3}}\ge c_2|\partial_qV(q)|^{\frac43}\ge c_2\kappa\,R_{V}^{^{\ge3}}(q)^4\ge c'_2\kappa\,R_{V}^{^{\ge3}}(q_j')^4~,\label{3.14M}
\end{align}
for any $q\in\mathrm{supp}\;\chi_j.$ 

Putting (\ref{3.13M}), (\ref{3.111M}) and (\ref{3.14M}) together,
\begin{align}
\|K_{V_{j}^2}u_j\|^2\ge C\Big(\|O_pu_j&\|^2+\frac{1}{4}\|\langle \partial_q V(q)\rangle^{2/3}u_j\|^2+\|\langle D_q\rangle^{2/3}u_j\|^2\nonumber\\&+ c_1\|\langle\mathrm{Hess}\;V(q)\rangle^{1/2}u_j\|^2+c'_2\kappa\, R_{V}^{^{\ge3}}(q_j')^4\|u_j\|^2+10\|t_j^2u_j\|^2\Big)~,
\end{align}
holds for all $\kappa\ge \kappa_2.$ 

\smallskip

\noindent\textbf{Case 2.} Assume $j\notin J(\kappa)$, with $\kappa\ge\kappa_2\ge\kappa_1\ge \kappa_0.$ Hence by Assumption~\ref{assumption1} (see (\ref{1.4})),  one has $$\mathrm{Tr}_{-,V}(q)\not=0\;\;\text{for all}\;\;q\in(\mathrm{supp}\;\chi_j)\cap(\mathbb{R}^d\setminus\Sigma(\kappa))\;\;\text{such that}\;\;|q|\ge C_1~.$$ In particular, since $|q_j'|\ge C(\kappa)\ge C_1,$ $$\mathrm{Tr}_{-,V_{j}^2}=\mathrm{Tr}_{-,V}(q_j')\not=0~.$$ Then referring again to \cite{BNV},
\begin{align*}
\|K_{V_{j}^2}u_j\|^2\ge c\,B_{V_{j}^2}\|u_j\|^2~,
\end{align*}
where \begin{align*}B_{V_{j}^2}&=\max
\Big(\min\limits_{q\in\mathbb{R}^d}|\nabla V_{j}^2(q)|^{4/3},\frac{1+\mathrm{Tr}_{-,V_{j}^2}}{\log(2+\mathrm{Tr}_{-,V_{j}^2})^2}\Big)
\\&=\max
\Big(\min\limits_{q\in\mathbb{R}^d}|\nabla V_{j}^2(q)|^{4/3},\frac{1+\mathrm{Tr}_{-,V}(q_j')}{\log(2+\mathrm{Tr}_{-,V}(q_j'))^2}\Big)\not=0~.\end{align*}
Hence we get in particular
\begin{align}
\|K_{V_{j}^2}u_j\|^2\ge \frac{1+\mathrm{Tr}_{-,V}(q'_j)}{\log(2+\mathrm{Tr}_{-,V}(q'_j))^2}\|u_j\|^2~.\label{3.155M}
\end{align}
Using again the condition (\ref{1.4}) in Assumption~\ref{assumption1}, there is a constant $C_1\ge1$ so that 
\begin{align*}
\frac{1}{2}\mathrm{Tr}_{-,V}(q'_j)\ge \frac{1}{2C_1}\mathrm{Tr}_{+,V}(q'_j)~,
\end{align*}
holds, which in turn implies
\begin{align}
\mathrm{Tr}_{-,V}(q'_j)\ge \frac{1}{2}\mathrm{Tr}_{-,V}(q'_j)+\frac{1}{2C_1}\mathrm{Tr}_{+,V}(q'_j)\ge \frac{1}{2C_1}(\mathrm{Tr}_{-,V}(q'_j)+\mathrm{Tr}_{+,V}(q'_j)) ~,\label{3.15M}
\end{align}
Then it follows from (\ref{3.155M}) and (\ref{3.15M})
\begin{align}
\|K_{V_{j}^2}u_j\|^2\ge c'\|\frac{\sqrt{1+|\mathrm{Hess}\;V(q'_j)|}}{\log(2+|\mathrm{Hess}\;V(q'_j)|)}u_j\|^2~.\label{3.100}
\end{align}
By Assumption~\ref{assumption1} (see  condition (\ref{1.5})) and
(\ref{3.100}), applying Lemma~\ref{lem2.1}, there is $\delta\in (0,1)$
and  $\Lambda_{\Sigma}(\varrho)$\,, $\lim_{\varrho\to
  +\infty}\Lambda_{\Sigma(\kappa)}(\varrho)=+\infty$ such that 
\begin{align*}
\frac{1+|\mathrm{Hess}\;V(q'_j)|}{\log(2+|\mathrm{Hess}\;V(q'_j)|)^2}&\ge\frac{1}{2^{2(1-\delta)}}(1+|\mathrm{Hess}\;V(q'_j)|)^{1-\delta}\nonumber\\
&\ge
 \frac{1}{4}|\mathrm{Hess}\;V(q'_j)|^{1-\delta}\\
&\ge
 \frac{\Lambda_{\Sigma(\kappa)}(|q_{j}'|)}{4}R_V^{^{\ge3}}(q'_j)^4
\geq \frac{\Lambda_{\Sigma(\kappa)}(C(\kappa))}{4}R_{V}^{^{\geq 3}}(q_{j}')^{4}\,.
\end{align*}
Therefore we get from the above inequality and (\ref{3.100}),
\begin{align}
\|K_{V_{j}^2}u_j\|^2\ge \Lambda_{\Sigma(\kappa)}(C(\kappa))R_{V}^{^{\ge3}}(q_j')^4\|u_j\|^2~.\label{3.110}
\end{align}
Next, remind that $t_j=2\langle \mathrm{Hess}\;V(q'_j)\rangle^{1/4}.$ By (\ref{B.21}) in Lemma~\ref{lem2.3}, the equivalence
\begin{align}t_j\asymp 2\langle \mathrm{Hess}\;V(q)\rangle^{1/4}~,\label{3.18M}
\end{align}
holds for any $q\in\mathrm{supp}\;\chi_j$ with $\left|q\right|\ge C(\kappa)\ge C_2(\kappa).$ From (\ref{2.6}) and (\ref{3.18M}) we see that
\begin{align}
\|K_{V_{j}^2}u_j\|^2+(1+10C)t_j^4\|u_j\|^2\ge C\Big(\|O_pu_j\|^2&+\|\langle\mathrm{Hess\;}V(q)\rangle^{1/2}u_j\|^2\nonumber\\&+\|\langle D_q\rangle^{2/3}u_j\|+9Ct_j^4\|u_j\|^2\Big)\;,\label{3.19M}
\end{align}
Since $j\notin J(\kappa)$
\begin{align}|\mathrm{Hess}\;V(q)|+R_{V}^{^{\ge3}}(q)^4+1\ge \frac{1}{\kappa}|\nabla V(q)|^{\frac43}\;,\label{3.20m}\end{align}
for all $q\in\mathrm{supp}\;\chi_j.$
Furthermore, it results by Lemma~\ref{lem2.1} that for all $q\in\mathrm{supp}\;\chi_j$ (where $j\notin J(\kappa)$) 
\begin{align}
|\mathrm{Hess}\;V(q)|+R_{V}^{^{\ge3}}(q)^4+1\le \frac{3}{2}|\mathrm{Hess}\;V(q)|~.\label{3.21m}
\end{align}
From (\ref{3.20m}) and (\ref{3.21m}) we get
\begin{align*}
|\mathrm{Hess}\;V(q)|\ge\frac{1}{2\kappa}|\nabla V(q)|^{\frac43}
\quad,\quad |\mathrm{Hess}\;V(q)|\ge \frac{2}{3}\geq \frac{1}{2\kappa}~.
\end{align*}
Hence there exists a constant $c''>0$ such that 
\begin{align*}\langle\mathrm{Hess}\;V(q)\rangle
\ge\frac{c''}{\kappa}\langle \partial_q V(q)\rangle^{4/3}~,\end{align*}
for any $q\in \mathrm{supp}\;\chi_j$  with $\left|q\right|\ge C(\kappa)\ge C_2(\kappa).$

The above inequality combined  with (\ref{3.19M})  leads to
\begin{align}\|K_{V_{j}^2}u_j\|^2+(1+10C) t_j^4\|u_j\|^2\ge C\Big(&\|O_pu_j\|^2+\|\langle D_q\rangle^{2/3}u_j\|^2+9t_j^4\|u_j\|^2\nonumber\\&+\frac{c''}{\kappa}\|\langle \partial_q V(q)\rangle^{2/3} u_j\|^2+\frac{1}{2}\|\langle\mathrm{Hess}\;V(q)\rangle^{1/2}u_j\|^2\Big)~,\label{3.130}
\end{align}
for all $\kappa\ge\kappa_2$ 

Collecting the estimates (\ref{3.130}) and (\ref{3.100}) we get
\begin{align}\log(t_j^4)^2\|K_{V_{j}^2}u_j\|^2\ge C"\Big(\|O_pu_j\|^2&+\|\langle D_q\rangle^{2/3}u_j\|^2+9t_j^4\|u_j\|^2\nonumber\\&+\frac{c''}{\kappa}\|\langle \partial_q V(q)\rangle^{2/3} u_j\|^2+\frac{1}{2}\|\langle\mathrm{Hess}\;V(q)\rangle^{1/2}u_j\|^2\Big)~.\label{3.2122M}
\end{align}
In order to reduce the written expressions we denote 
\begin{align*}
\Lambda_{1,j}&=\frac{O_p}{\log(t_j^4)}~,\quad\Lambda_{2,j}=\frac{\langle\mathrm{Hess}\;V(q)\rangle^{1/2}}{\log(t_j^4)}~,\quad\Lambda_{3,j}=\frac{\langle \partial_q V(q)\rangle^{2/3}}{\log(t_j^4)}~,\quad\Lambda_{4,j}=\frac{t_j^2}{\log(t_j^4)}~.
\end{align*}
The estimate (\ref{3.2122M}) can be rewritten as follows
\begin{align}\|K_{V_{j}^2}u_j\|^2\ge C"\Big(\|\Lambda_{1,j}u_j\|^2+\frac{1}{2}\|\Lambda_{2,j}u_j\|^2+\frac{c''}{2\kappa}\|\Lambda_{3,j}u_j\|^2+9\|\Lambda_{4,j}u_j\|^2+\|\frac{\langle D_q\rangle^{2/3}}{\log(t_j^4)}u_j\|^2\Big)~.\label{3.140}
\end{align}
Using (\ref{3.140}) and (\ref{3.110}) we obtain
\begin{align*}(1+C")\|K_{V_{j}^2}u_j\|^2\ge C"\Big(\|\Lambda_{1,j}u_j\|^2&+\frac{1}{2}\|\Lambda_{2,j}u_j\|^2+\frac{c''}{2\kappa}\|\Lambda_{3,j}u_j\|^2+9\|\Lambda_{4,j}u_j\|^2\\&+\|\frac{\langle D_q\rangle^{2/3}}{\log(t_j^4)}u_j\|^2+\Lambda_{\Sigma(\kappa)}(C(\kappa))R_{V}^{^{\ge3}}(q_j')^4\|u_j\|^2\Big)~.
\end{align*}
Therefore in both cases, that is for all $j\in I(\kappa)$ where $\kappa\ge \kappa_2$
\begin{align}\|K_{V_{j}^2}u_j\|^2\ge C^{(3)}\Big(\|\Lambda_{1,j}u_j\|^2&+\|\Lambda_{2,j}u_j\|^2+\frac{1}{\kappa}\|\Lambda_{3,j}u_j\|^2+\|\Lambda_{4,j}u_j\|^2\nonumber\\&+\|\frac{\langle D_q\rangle^{2/3}}{\log(t_j^4)}u_j\|^2+\min\Big(\kappa,\Lambda_{\Sigma(\kappa)}(C(\kappa))\Big)R_{V}^{^{\ge 3}}(q'_j)^4\|u_j\|^2\Big)~.\label{3.160}
\end{align}
Due to the elementary inequality $u^{4/3}+v^4\ge \frac{1}{c_{0}}(u^2+v^4)^{2/3}$ satisfied for all $u,v\ge 1$, we obtain for all $\kappa\ge \kappa_2$
\begin{align}
\|\frac{\langle D_q\rangle^{2/3}}{\log(t_j^4)}u_j\|^2+\frac{1}{2}\min\Big(\kappa,\Lambda_{\Sigma(\kappa)}(C(\kappa))\Big)R_{V}^{^{\ge 3}}(q'_j)^4\|u_j\|^2\ge \frac{1}{c_{0}}\|\Lambda_{5,j}u_j\|^2~,\label{3.233M}\end{align}
where \begin{align*}
\Lambda_{5,j}=\frac{(1+D_q^2+R_{V}^{^{\ge3}}(q'_j)^4)^{\frac{1}{3}}}{\log(t_j^4)}~.
\end{align*}
In conclusion, we get by (\ref{3.160}) and (\ref{3.233M}) for every $j\in I(\kappa)$ with $\kappa\ge \kappa_2$
\begin{align*}\|K_{V_{j}^2}u_j\|^2\ge C^{(3)}\Big(\|\Lambda_{1,j}u_j\|^2+\|\Lambda_{2,j}u_j\|^2&+\frac{1}{\kappa}\|\Lambda_{3,j}u_j\|^2+\|\Lambda_{4,j}u_j\|^2\nonumber\\& +\frac{1}{c_{0}}\|\Lambda_{5,j}u_j\|^2+\frac{1}{2}\min\Big(\kappa,\Lambda_{\Sigma(\kappa)}(C(\kappa))\Big)R_{V}^{^{\ge 3}}(q'_j)^4\|u_j\|^2\Big)~.
\end{align*}
We now fix the choice firstly of $C(\kappa)$ and secondly of
$\kappa$\,.
Because $\lim_{\varrho\to
  +\infty}\Lambda_{\Sigma(\kappa)}(\varrho)=+\infty$\,, we can choose
for any $\kappa\geq \kappa_{2}$\,,
$C(\kappa)\geq \max (C_{1},C_{2}(\kappa))$ such that
$\Lambda_{\Sigma(\kappa)}(C(\kappa))\geq \kappa$\,. 
We then choose $\kappa=\kappa_{3}\geq \kappa_{2} $ such that
$$
\frac{C^{(3)}}{8}\min\Big(\kappa_3,\Lambda_{\Sigma(\kappa_{3})}(\kappa_3)\Big)=
\frac{C^{(3)}\kappa_{3}}{8}\ge (c'_{d,r})^2~,$$
where  $c'_{d,r}$ is the constant in (\ref{3.666M}),
\begin{align}
\sum\limits_{j\in I(\kappa_3)}\frac{1}{4}\|K_{V_{j}^2}u\|^2-(c'_{d,r})^2R_{V}^{^{\ge 3}}(q'_j)^4\|u_j\|^2\ge \frac{C^{(3)}}{8}&\sum \limits_{j\in I(\kappa_3)}\Big(\|\Lambda_{1,j}u_j\|^2+\|\Lambda_{2,j}u_j\|^2\nonumber\\&+\frac{1}{\kappa_3}\|\Lambda_{3,j}u_j\|^2+\|\Lambda_{4,j}u_j\|^2 +\frac{1}{c_{0}}\|\Lambda_{5,j}u_j\|^2\Big)~.\label{3.227M}
\end{align}
\smallskip

\noindent\textbf{Step 2, $\mathbf{j\not\in I(\kappa_{3})}$:}
The set  $\mathbb{N}\setminus I(\kappa_{3})$ is  now a fixed finite  set and
we can define
\begin{multline*}
   C^{(4)}=\max_{j\in \mathbb{N}\setminus
   I(\kappa_{3})}
\left[A_{V_{j}^{2}}+\sup_{q\in \mathrm{supp}\,\chi_{j}}\Big(\langle
\mathrm{Hess}~V(q)\rangle+\langle \partial_{q}V(q)\rangle^{4/3}\Big)
\right.
\\
\left.
+\frac{t_{j}^{4}}{\log(t_{j}^{4})^{2}}+(1+(c_{d,r}')^{2})(1+R_{V}^{\geq
3}(q_{j}'))^{4}\right]\,.
\end{multline*}
From the lower bound \eqref{eq4} we deduce
\begin{multline*}
\frac{1}{4}\|K_{V_{j}^{2}}u_{j}\|+C^{(4)}\|u_{j}\|^{2}-(c_{d,r}')^{2}R_{V}^{^{\geq
3}}(q_{j}')^{4}\|u_{j}\|^{2}
\geq\frac{c}{4}\left[\|O_pu_{j}\|^{2}+\|\langle
  D_{q}\rangle^{2/3}u_{j}\|^{2}\right] +(1+R_{V}^{^{\geq
3}}(q_{j}'))^{4}\|u_{j}\|^{2} \\
+\|\langle \partial_{q}V(q)\rangle^{2/3}u_{j}\|^{2}
+\|\langle 
\mathrm{Hess}~V(q)\rangle^{1/2}u_{j}\|^{2}
+\|\frac{t_{j}^{2}}{\log(t_{j}^{4})}u_{j}\|^{2}
\end{multline*}
With the quantities 
\begin{align*}
\Lambda_{1,j}=\frac{O_p}{\log(t_j^4)}~,&\quad\Lambda_{2,j}=\frac{\langle\mathrm{Hess}\;V(q)\rangle^{1/2}}{\log(t_j^4)}~,\quad\Lambda_{3,j}=\frac{\langle \partial_q V(q)\rangle^{\frac23}}{\log(t_j^4)}~,\\&\quad\Lambda_{4,j}=\frac{t_j^2}{\log(t_j^4)}\;,\quad\Lambda_{5,j}=\frac{(1+D_q^2+R_{V}^{^{\ge3}}(q'_j)^4)^{\frac{1}{3}}}{\log(t_j^4)}~.
\end{align*}
where $t_{j}\geq 2$ we deduce
\begin{multline}
 \sum\limits_{j\not\in
  I(\kappa_3)}\frac{1}{4}\|K_{V_{j}^2}u_j\|^2-(c'_{d,r})^2R_{V}^{^{\ge
  3}}(q'_j)^4\|u_j\|^2+C^{(4)}\|u_j\|^2
\geq
\\
C^{(5)}\sum\limits_{j\not\in
  I(\kappa_3)}\Big(\|\Lambda_{1,j}u_j\|^2+\|\Lambda_{2,j}u_j\|^2+\frac{1}{\kappa_3}\|\Lambda_{3,j}u_j\|^2+\|\Lambda_{4,j}u_j\|^2
+\frac{1}{c_{0}}\|\Lambda_{5,j}u_j\|^2\Big)
,\label{3.229M}
\end{multline}
Collecting (\ref{3.666M}), (\ref{3.227M}) and (\ref{3.229M}), there
exists a positive constant $C^{(6)}\geq 1$ depending on $V$ such that
\begin{align}
\|K_{V}u\|^2_{L^2}+C^{(6)}\|u\|^2_{L^2}\ge \frac{1}{C^{(6)}}\sum\limits_{j\in \mathbb{N}}\Big(\|\Lambda_{1,j}u_j\|^2+\|\Lambda_{2,j}u_j\|^2&+\|\Lambda_{3,j}u_j\|^2\nonumber\\&+\|\Lambda_{4,j}u_j\|^2 +\|\Lambda_{5,j}u_j\|^2\Big)~.
\end{align}
\smallskip

\noindent\textbf{Step 3.} In this final step, set $L(s)=\frac{s+1}{\log(s+1)}$ for all $s\ge 1.$ Notice that there exists a constant $c>0$ such that for all $x\ge 1$,
\begin{align*}
\inf\limits_{t\ge 2}\frac{x}{\log(t)}+t\ge \frac{1}{c}L(x)~.
\end{align*}
In view of the above estimate,
\begin{align*}
\|\Lambda_{1,j}u_j\|^2+\frac{1}{4}\|\Lambda_{4,j}u_j\|^2&\ge\frac{1}{4}\int \frac{\lambda^2}{(\log(t_j^4))^2}+t_j^2\;d\mu_{u_j}(\lambda)\\&\ge \frac{1}{8}\int(\frac{\lambda}{\log(t_j)}+t_j)^2d\mu_{u_j}(\lambda)\\&\ge\frac{1}{c_{3}}\|L(O_p)u_j\|^2~.
\end{align*}
Summing over $j$, we obtain the first term in the desired estimation $(\ref{1.6}).$ Likewise for the second term
 \begin{align*}
 \|\Lambda_{3,j}u_j\|^2+\frac{1}{4}\|\Lambda_{4,j}u_j\|^2\ge\frac{1}{c_{4}}\|L(\langle \partial_q V(q)\rangle^{2/3})u_j\|^2~,
 \end{align*}
 with \begin{align*}\sum\limits_{j\in\mathbb{N}}\|L(\langle
        \partial_q V(q)\rangle^{2/3})u_j\|^2=\|L(\langle \partial_q V(q)\rangle^{2/3})u\|^2~.\end{align*}
 To obtain the third term in $(\ref{1.6})$ write samely
 \begin{align*}
 \|\Lambda_{2,j}u_j\|^2+\frac{1}{4}\|\Lambda_{4,j}u_j\|^2\ge\frac{1}{c_{5}}\|L(\langle \mathrm{Hess}\;V(q)\rangle^{1/2})u_j\|^2~,
 \end{align*}
 with \begin{align*}\sum\limits_{j\in\mathbb{N}}\|L(\langle\mathrm{Hess}\; V(q)\rangle^{1/2})u_j\|^2=\|L(\langle \mathrm{Hess}\; V(q)\rangle^{1/2})u\|^2~.\end{align*}
Doing similarly again
 \begin{align*}
\|\Lambda_{5,j}u_j\|^2+\frac{1}{4}\|\Lambda_{4,j}u_j\|^2\ge\frac{1}{c_{6}}\|L((1+D_q^2+R_{V}^{^{\ge3}}(q'_j)^4)^{\frac{1}{3}})u_j\|^2~.
\end{align*}
By Lemma~\ref{lem2.2} we get \begin{align*}\sum\limits_{j\in\mathbb{N}}\|L((1+D_q^2+R_{V}^{^{\ge3}}(q'_j)^4)^{\frac{1}{3}})u_j\|^2\ge \frac{1}{c_{6}}\|L((1+D_q^2+R_{V}^{^{\ge3}}(q)^4)^{\frac{1}{3}})u\|^2~,\end{align*}
To conclude, just remark that 
\begin{align*}
\langle u,(1+D_q^2+R^{^{\ge3}}_{V}(q)^4)u\rangle\ge \langle
  u,(1+D_q^2)u\rangle\ge \langle u,\langle D_q^{2}\rangle u\rangle > \|u\|^2
\end{align*}
holds for all $u\in\mathcal{C}_0^{\infty}(\mathbb{R}^{2d}),$ then applying (\ref{2.144M}) in Lemma~\ref{lem2.3M} with $A=(1+D_q^2+R^{^{\ge3}}_{V}(q)^4)$, $B=\langle D_q^{2}\rangle,$ $\alpha_0=\frac{2}{3}$ and $k=2$ we obtain
\begin{align*}
\|L((1+D_q^2+R_{V}^{^{\ge3}}(q)^4)^{\frac{1}{3}})u\|^2\ge
  \|L(\langle D_q^{2}\rangle^{\frac{1}{3}})u\|^2\geq\frac{1}{c_{7}}\|L(\langle
  D_{q}\rangle^{2/3})u\|^{3}
\end{align*}
for all $u\in\mathcal{C}_0^{\infty}(\mathbb{R}^{2d}).$
 
Finally collecting all terms, we have found $C_{V}\geq 1$ such that
\begin{align}
\|K_{V}u\|^2_{L^2}+C_{V}\|u\|^2_{L^2}\ge \frac{1}{C_{V}}\Big(\|L(O_p)u\|^2_{L^2}&+\|L(\langle\nabla V(q)\rangle^{\frac{2}{3}})  u\|^2_{L^2}\nonumber\\&+\|L(\langle\mathrm{Hess}\; V(q)\rangle^{\frac{1}{2}} ) u\|^2_{L^2}+\|L(\langle D_q\rangle^{\frac{2}{3}} ) u\|^2_{L^2}\Big)
\end{align}
holds for all $u\in\mathcal{C}_0^{\infty}(\mathbb{R}^{2d})$\,. Because
${\cal C}^{\infty}_{0}(\mathbb{R}^{2d})$ is dense in $D(K_{V})$
endowed with the graph norm, the result extends to any $u\in D(K_{V})$\,.
\end{proof}
 \section{Applications}
This section is devoted to some applications of Theorem~\ref{thm1.1}. In each of the following examples we examine that the Assumption~\ref{assumption1} is well fulfilled.\\

\noindent\textbf{Example 1:} 
Let us consider as a first example of application the case \begin{align*}V(q_1,q_2)=-q_1^2q_2^2,\;\text{with}\; q=(q_1,q_2)\in\mathbb{R}^2~,\end{align*}
By direct computation 
\begin{align*}\partial_qV(q)=\begin{pmatrix}
-2q_1q_2^2\\-2q_2q_1^2
\end{pmatrix}\;,~ |\partial_qV(q)|=2|q_1q_2||q|~,\end{align*}

\begin{align*}\mathrm{Hess}\;V(q)=\begin{pmatrix}
-2q_2^2&-4q_1q_2\\-4q_1q_2&-2q_1^2
\end{pmatrix}\;,\;|\mathrm{Hess}\;V(q)|=2\sqrt{|q|^4+6q_1^2q_2^2}\asymp |q|^2~,\end{align*}

\begin{align*}R_{V}^{^{\ge 3}}(q)=|4q_2|^{1/3}+|4q_1|^{1/3}+2\times4^{1/4}~.\end{align*}
It is clear that the trace of $\mathrm{Hess}\;V(q)$ given by $-2|q|^2$ is stricly negative for all $q\in\mathbb{R}^2.$ Hence 
\begin{align*}
\mathrm{Tr}_{-,V}(q)\ge \mathrm{Tr}_{+,V}(q)\;\;\text{ for all }\;q\in\mathbb{R}^2~.
\end{align*}
Moreover, for all $\kappa>0$ the algebraic set $\mathbb{R}^2\setminus \Sigma(\kappa)$ is not bounded since $(0,q_2)\in \mathbb{R}^2\setminus \Sigma(\kappa)$ for all $q_2\in \mathbb{R}.$ Furthermore for $\kappa>1$ chosen as we want
\begin{align*}
\lim\limits_{\substack{q\to\infty \\ q\in \mathbb{R}^2\setminus \Sigma(\kappa)}}\frac{R_{V}^{^{\ge 3}}(q)^4}{|\mathrm{Hess}\;V(q)|}=\lim\limits_{\substack{q\to\infty \\ q\in \mathbb{R}^2\setminus \Sigma(\kappa)}}\frac{|q|^{4/3}}{|q|^{2}}=0~,
\end{align*}
since $R_{V}^{^{\ge 3}}(q)^4\le |q|^{4/3}$ when $|q|\ge 2^3\times4^{3/4}.$

Below we sketch as example $\Sigma(800)$ in a blue color. 
\begin{figure}[H]
	\centering
\includegraphics[scale=0.3]{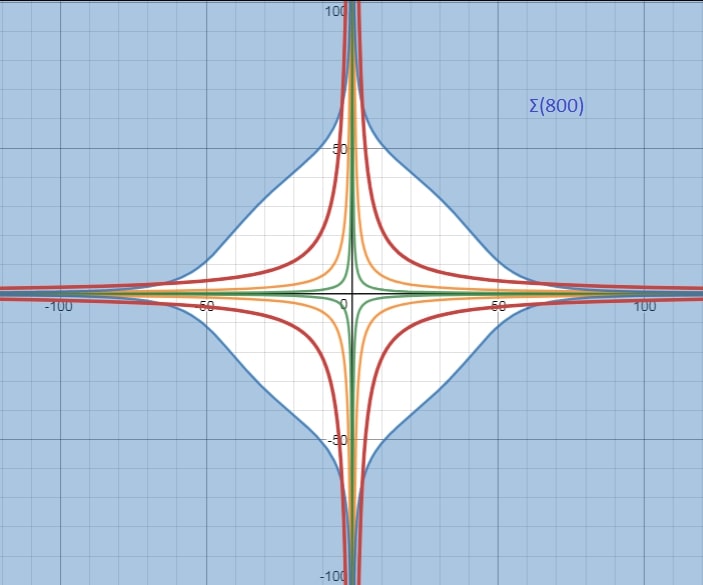}
	\caption{Contour lines of $V(q_1,q_2)=-q_1^2q_2^2$}
	\label{fig:n}
\end{figure}

\noindent\textbf{Example 2:} Let $n\in \mathbb{N}.$ The polynomial $V(q)=-q_1^2(q_1^2+q_2^2)^n$ verifies the Assumption~\ref{assumption1} only for $n=1.$
\\
A straight forward computation shows that
\begin{align*}\partial_qV(q)=-\begin{pmatrix}
2q_1(|q|^{2n}+nq_1^2|q|^{2(n-1)})\\2nq_2q_1^2|q|^{2(n-1)}
\end{pmatrix}~,\end{align*}
\begin{align*}\mathrm{Hess}\;V(q)=-2|q|^{2(n-2)}\begin{pmatrix}
|q|^{4}+5nq_1^2|q|^{2}+2n(n-1)q_1^4&     2nq_1q_2|q|^{2}+2n(n-1)q_1^3q_2 \\2nq_1q_2|q|^{2}+2n(n-1)q_1^3q_2&nq_1^2|q|^{2}+2n(n-1)q_1^2q_2^2
\end{pmatrix}~.\end{align*}
Notice that the trace of $\mathrm{Hess}\;V(q)$ equals \begin{align*}-2|q|^{2(n-2)}\Big(|q|^{4}+5nq_1^2|q|^{2}+2n(n-1)q_1^4+nq_1^2|q|^{2}+2n(n-1)q_1^2q_2^2\Big)\le0~,\end{align*}  for all $q\in\mathbb{R}^2.$ Hence \begin{align*}-\mathrm{Tr}_{-,V}(q)+\mathrm{Tr}_{+,V}(q)\le 0,\;\;\text{for any}\;\;q\in \mathbb{R}^2~.\end{align*}
In addition for all $\kappa>0$ the set $\mathbb{R}^2\setminus \Sigma(\kappa)$ is not bounded since $(0,q_2)\in\mathbb{R}^2\setminus \Sigma(\kappa)$ for all $q_2\in \mathbb{R}.$
\\
For $q$ large enough $|\mathrm{Hess}\;V(q)|\asymp |q|^{2n}$ and  $|D^3V(q)|\asymp |q|^{2n-1}$ then $$\frac{R_{V}^{^{\ge 3}}(q)^4}{|\mathrm{Hess}\;V(q)|}\asymp\frac{(|q|^{2n-1})^{4/3}}{|q|^{2n}}~.$$ 
Hence  \begin{align*}
\lim\limits_{\substack{q\to\infty \\ q\in \mathbb{R}^2\setminus \Sigma(\kappa)}}\frac{R_{V}^{^{\ge 3}}(q)^4}{|\mathrm{Hess}\;V(q)|}=0~\;\;\;\text{if and only if}\;\;\;n<2~.\end{align*} 
Taking as example $\kappa=800,$ we get the following shape of $\Sigma(800)$ colored in blue.

\begin{figure}[H]
	\centering
\includegraphics[scale=0.45]{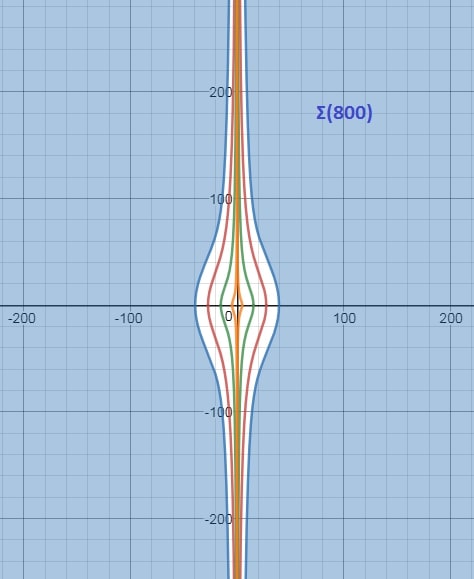}
	\caption{Contour lines of $V(q_1,q_2)=-q_1^2(q_1^2+q_2^2)$}
	\label{fig:n}
\end{figure}

\noindent\textbf{Example 3:} For $\epsilon\in\mathbb{R}\setminus\lbrace0,-1\rbrace,$ we consider $V(q_1,q_2)=(q_1^2-q_2)^2+\epsilon q_2^2.$
For all $q\in\mathbb{R}^2$ one has
\begin{align*}\partial_qV(q)=\begin{pmatrix}
4q_1(q_1^2-q_2)\\-2(q_1^2-q_2)+2\epsilon q_2
\end{pmatrix}~,\;|\partial_qV(q)|=4|q_1(q_1^2-q_2)|+|-2(q_1^2-q_2)+2\epsilon q_2|~,\end{align*}
\begin{align*}\mathrm{Hess}\;V(q)=\begin{pmatrix}
12q_1^2-4q_2&-4q_1\\-4q_1&2(1+\epsilon)
\end{pmatrix}~,\;|\mathrm{Hess}\;V(q)|=|12q_1^2-4q_2|+8|q_1|+4|1+\epsilon|~,\end{align*}
\begin{align*}R_{V}^{^{\ge 3}}(q)=(24|q_1|)^{1/3}+3\times4^{1/3}+24^{1/4}~.\end{align*}
In this case, we are going to show that for all $\kappa>0$ the algebraic set $\mathbb{R}^2\setminus\Sigma(\kappa)$ is bounded. Let $(q_1,q_2)\in\mathbb{R}^2\setminus\Sigma(\kappa)$ then 
\begin{align*} \Big(|\mathrm{Hess}\; V(q)|+R_{V}^{^{\ge 3}}(q)^4+1\Big)\ge\frac{1}{\kappa}|\nabla V(q)|^{\frac{4}{3}}~.\end{align*}
Up to a change of coordinates $X_1=q_1,\;X_2=q_1^2-q_2$ the above inequality is equivalent to
\begin{align*} \Big(4|2X_1^2+X_2|+8|X_1|+4|1+\epsilon|&+\Big((24|X_1|)^{1/3}+3\times4^{1/3}+24^{1/4}\Big)^4+1\Big)\\&\ge \frac{1}{\kappa}\Big(4|X_1X_2|+|-2(1+\epsilon)X_2+2\epsilon X_1^2|\Big)^{\frac{4}{3}}~.\end{align*}
Using the triange inequality in the right hand side and the reverse triangle inequality with the elementary inequality $(u+v)^{\frac{4}{3}}\ge u^{\frac{4}{3}}+v^{\frac{4}{3}}$ satisfied for all $u,v\ge0,$ it follows that 
\begin{align}
|X_1|^2+|X_2|+|X_1|+\Big(|X_1|^{\frac{1}{3}}+c\Big)^4\ge \frac{c'}{\kappa}\Big(\Big||2(1+\epsilon)X_2|-|2\epsilon X_1^2|\Big|^{\frac{4}{3}}+|X_1X_2|^{\frac{4}{3}}\Big)~.\label{4.1M}
\end{align}
Suppose first that $|X_1|\le 1.$ The inequality (\ref{4.1M}) implies 
\begin{align}
|X_2|+c_1\ge \frac{c'}{\kappa}\Big||2(1+\epsilon)X_2|-|2\epsilon X_1^2|\Big|^{\frac{4}{3}}~.\label{4.2M}
\end{align}
The right hand part in the above inequality is upper bounded by $|X_2|+c_1$ where $c_1$ is some positive constant. Now we distinguish two case:
\\
\textbf{Case 1:} If $\frac{1}{2}|2(1+\epsilon)X_2|\le |2\epsilon X_1^2|$ or equivalently $|X_2|\le|\frac{2\epsilon}{1+\epsilon}||X_1^2|$ then $|X_2|\le |\frac{2\epsilon}{1+\epsilon}|~.$ 
\\
\textbf{Case 2:} Else if $\frac{1}{2}|2(1+\epsilon)X_2|\ge |2\epsilon X_1^2|$ then  we get
\begin{align*}|X_2|+c_1\ge \frac{c'}{\kappa}|1+\epsilon||X_2|^{4/3}~.\end{align*}
Using the fact that $\epsilon\not=-1$, we deduce that $X_2$ must be also bounded.

Now if $|X_1|\ge 1,$ we derive from (\ref{4.1M}) the following esimates 
\begin{align}
|X_1|^2+|X_2|+c_3\ge \frac{c_4}{\kappa}\Big||2(1+\epsilon)X_2|-|2\epsilon X_1^2|\Big|^{\frac{4}{3}}~,\label{4.3M}
\end{align}
\begin{align}
|X_1|^2+|X_2|+c_3\ge \frac{c_4}{\kappa}|X_1X_2|^{\frac{4}{3}}~.\label{4.4M}
\end{align}
Here we study three cases. 
\\
$\bullet$ Firstly if  $\frac{1}{2}|2(1+\epsilon)X_2|\ge |2\epsilon X_1^2|$ or equivalently $|X_1|\le|\frac{1+\epsilon}{2\epsilon}||X_2|$ then (\ref{4.3M}) gives
\begin{align*}
(1+|\frac{1+\epsilon}{\epsilon}|)|X_2|+c_3\ge \frac{c_4}{\kappa}|(1+\epsilon)X_2|^{\frac{4}{3}}~.
\end{align*}
Since $\epsilon\not=-1,$ it follows that $X_2$ is bounded and so is $X_1.$
\\
$\bullet$ Now if $2|2(1+\epsilon)X_2|\le |2\epsilon X_1^2|$ or samely $|X_2|\le|\frac{\epsilon}{2(1+\epsilon)}||X_1^2|$ the estimates (\ref{4.3M}) leads to
\begin{align*}
(1+|\frac{\epsilon}{2(1+\epsilon)}|)|X_1|^2+c_3\ge \frac{c_4}{\kappa}|\epsilon X_1|^{\frac{8}{3}}~.
\end{align*}
Since $\epsilon\not=0,$ it follows that $X_1$ is bounded and so is $X_2.$
\\
$\bullet$ Finally if $\frac{1}{2}|2(1+\epsilon)X_2|\le |2\epsilon X_1^2|\le 2|2(1+\epsilon)X_2|,$
then by (\ref{4.4M})
\begin{align*}
(1+|\frac{2\epsilon}{1+\epsilon}|)|X_1|^2+c_3\ge \frac{c_4}{\kappa}\Big(| X_1||\frac{\epsilon}{2(1+\epsilon)}|X_1^2|\Big)^{\frac{4}{3}}~.
\end{align*}
Hence since $\epsilon\not=0,$ $X_1$ is bounded and then $X_2$ is so.

Below we sketch as example $\Sigma(2)$ in a blue color.

\begin{figure}[H]
	\centering
\includegraphics[scale=0.5]{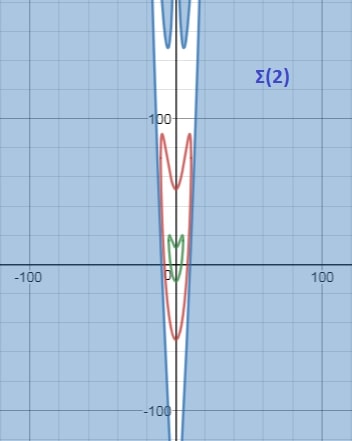}
	\caption{Contour lines of $V(q_1,q_2)=(q_1^2-q_2)^2+0.5 q_2^2~.$}
	\label{fig:n}
\end{figure}
For $\epsilon=0,$ thanks to $\cite{HeNi}$ (see Proposition 10.21 page 111), we know that the Witten Laplacian defined by
$$\Delta_V^{(0)}=-\Delta_{q}+\left|\nabla V(q)\right|^{2}-\Delta V(q)~,\;\;q=(x_1,x_2)\in\mathbb{R}^2$$ has no compact resolvent and then the Kramers-Fokker-Planck operator $K_V$ has no compact resolvent.

This example was studied in the case of the Witten Laplacian operator by B.Helffer and F.Nier in their book \cite{HeNi}. A small mistake was done in \cite{HeNi} in Proposition 10.20. In fact the equations $l_{11} = l_{12}=l_{111}= 0$ should be replaced 
by $(1+ \epsilon)l_{11} =l_{12}=l_{111} = 0.$ When $\epsilon=-1,$ we can eventually construct a Weyl sequence for the Witten Laplacian operator in the following way.
In this case the potential $V(q_1,q_2)=(q_1^2-q_2)^2-q_2^2$ is equal to $-2q_2 q_1^2+q_1^4.$

In order to construct a Weyl sequence for  $\Delta_V^{(0)},$ it is sufficient to take $\chi(\frac{(q_2+n^2)}{n})$ where $\chi$ is a cutoff function supported in $[-1,1]$ and then consider the sequence

$$u_n(q_1,q_2)=\chi(\frac{(q_2+n^2)}{n}) \exp(-V(q_1,q_2))~.$$

The support of $u_n$ is then included in $-n^2-n\leq q_2 \leq -n^2+n.$ Hence the $u_n$'s have disjoint supports for large $n.$

Therefore we have  $$-2n^2\leq q_2\leq -\frac{n^2}{2}\quad \text{and}\quad -4n^2 q_1^2-q_1^4 \leq -V(q_1,q_2)\leq -n^2 q_1^2-q_1^4\leq -n^2 q_1^2~.$$

As a result, we get for $n$ large 
 \begin{align*}\frac{\langle u_n, \Delta_V^{(0)} u_n\rangle}{\|u_n\|^2}& =\frac{\|(\partial_q+\partial_qV(q))(u_n)\|^2}{\|u_n\|^2}\\&=\frac{\|(\partial_q\chi)e^{-V}\|^2}{\|u_n\|^2}=O(\frac{1}{n^2})~.
 \end{align*}

Here to get the lower bound of the  the above quantity we restrict the integral in $q_1=O(\frac{1}{n}).$ As a conclusion, for $\epsilon=-1$ the Witten Laplacian attached to $V(q_1,q_2)=q^2_1q^2_2+ \epsilon(q^2_1+ q^2_2)$ has no compact resolvent and then the Kramers-Fokker-Planck operator $K_V$ has no compact resolvent.

\begin{appendices}

 \section{Slow metric, partition of unity }\label{B}
 The purpose of this appendix is to state with references or proofs the
facts concerning metrics which are needed in the article.
We first remind the following definitions.
\begin{defis}A metric $g$ on $\mathbb{R}^m$ is called a slowly varying metric if there exists a constant $C\ge 1$ such that for all $x,y\in\mathbb{R}^m$ satisfying $g_x(x-y,x-y)\le C^{-1}$ it follows that 
\begin{align}
C^{-1}g_x(z,z)\le g_y(z,z)\le Cg_x(z,z)\label{A.1}
\end{align}
holds for all $z\in\mathbb{R}^m.$

Let $g^1$ and $g^2$ be two metrics. We say that $g^1$ is $g^2$ slow if there is a constant $c\ge 1$ such that for all $x,y\in\mathbb{R}^m$
\begin{align}
g^2_x(x-y,x-y)\le c^{-1}\Rightarrow c^{-1}g^1_x(z,z)\le g^1_y(z,z)\le cg^1_x(z,z)~.\label{def}
\end{align} 
holds for all $z\in\mathbb{R}^m.$
\begin{RQ}
The second statement in the above definitions is a typical application of the notion of the second microlocalisation developed by Bony-lerner (see \cite{BoLe}).
\end{RQ}
\end{defis}
\begin{RQ}
 The property~\ref{A.1} will be satisfied if we ask only that 
 \begin{align}
 \exists C\ge 1, \forall x,y,z\in \mathbb{R}^{m},\;g_x(x-y)\le C^{-1}\Longrightarrow g_y(z)\le Cg_x(z)~.\label{A.2}
\end{align}
Indeed, assuming (\ref{A.2}) gives that wherever $g_x(x-y)\le C^{-1}$ (which is less than or equal to one since $C\ge1$ from (\ref{A.2})  with x=y) this implies $g_y(y-x)\le C^{-1}$ and then\newline $g_x(z)\le Cg_y(z)$, so that (\ref{A.1}) is well satisfied.
 \end{RQ}

 \begin{NTS}
 For $r\in \mathbb{N}$, let $E_r$ denote the set of polynomials with degree not greater than~$r$:
 \begin{align*}
 E_r=\left\{P\in \mathbb{R}[X_1,...,X_d],\; d^{\circ}P\le r\right\}~.
 \end{align*}

 For a polynomial $P\in E_r$ and $n\in\left\{1,...,r\right\}$, the function $R^{^{\ge n}}_P:\mathbb{R}^d\to \mathbb{R}$ is defined by
 \begin{align}
R^{^{\ge n}}_P(q)=\sum\limits_{n\le |\alpha|\le r}|\partial^{\alpha}_qP(q)|^{\frac{1}{|\alpha|}}\;.\label{A.3}
 \end{align}
 \end{NTS}
 
In the present article we are mainly concerned with the metric $g^n=
R_{P}^{^{\ge n}}(q)^2\,dq^2$ where $n\in \left\{1,...,r\right\}$ which satisfies the following properties.
\begin{lem}\label{A.5} Let $n$ a natural number in $\left\{1,...,r\right\}.$
\\
1) The metric $g^n$ is slow: There exists a uniform $C=C(n,r,d)\ge 1$ such that 
\begin{align}
R_{P}^{^{\ge n}}(q)|q-q'|\le C^{-1}\Longrightarrow \Big(\frac{R_{P}^{^{\ge n}}(q)}{R_{P}^{^{\ge n}}(q')}\Big)^{\pm1}\le C\label{A.55}
\end{align}
2) The metric $g^{n-1}$ is $g^n$ slow: There is a constant $C'=C'(n,r,d)\ge 1$ so that 
\begin{align}
R_{P}^{^{\ge n}}(q)|q-q'|\le C'^{-1}\Longrightarrow \Big(\frac{R_{P}^{^{\ge n-1}}(q)}{R_{P}^{^{\ge n-1}}(q')}\Big)^{\pm1}\le C'
\end{align}
\end{lem}
\begin{proof}
Assume $n,r\in\mathbb{N}^*$ with $n\ge r.$
Consider the map \begin{align*}f:&E_r\to E_r/E_n;\;\\&P\mapsto \overline{P}:\mathbb{R}^d\to\mathbb{R}\\&\quad\quad\quad\quad x\mapsto \overline{P}(x)\sum\limits_{n\le |\alpha|\le r}\frac{\partial^{\alpha}_xP(0)}{\alpha!}\Big(\frac{x}{R_P^{^{\ge n}}(0)}\Big)^{\alpha}~.
\end{align*} Set $K_{n,r}:=f(E_r)=\left\{\overline{P}\in E_r/E_n,\; R_{\overline{P}}^{^{\ge n}}(0)=1\right\}.$ Assume $\overline{P}\in K_{n,r}$ and $\beta\in\mathbb{N}^d$ with $|\beta|\ge n.$ Notice that there is a constant $c\ge1$ (uniform with respect to $\overline{P}$ and $\beta$) such that for $|t|\le c^{-1},$
\begin{align}
|\partial_t^{\beta}\overline{P}(t)-\partial_t^{\beta}\overline{P}(0)|&=|\sum\limits_{|\alpha|=1}\frac{\partial_t^{\alpha+\beta}\overline{P}(0)}{\alpha!}t^{\alpha}|\nonumber\\&\le \sum\limits_{|\alpha|=1}|\frac{\partial_t^{\alpha+\beta}\overline{P}(0)}{\alpha!}|\;|t|\le d|t|~.
\end{align}
On the other hand, the application $$\mathbb{R}^{\mathbb{N}}\to \mathbb{R}^{\mathbb{N}},\; (u_{\beta})_{n\le |\beta|\le r}\mapsto (|u_{\beta}|^{\frac{1}{|\beta|}})_{n\le |\beta|\le r}$$ is continuous. Then for all $\delta>0$ there exists $\eta=\eta(n,r)>0$ so that
\begin{align}
\max\limits_{n\le |\beta|\le r}|u_{\beta}-v_{\beta}|\le \eta\Longrightarrow \sum\limits_{n\le |\beta|\le r}\Big||u_{\beta}|^{\frac{1}{|\beta|}}-|v_{\beta}|^{\frac{1}{|\beta|}}\Big|\le \delta\;.
\end{align}
Thus for all $\delta>0$ there is a strictly positive constant $C_1=C_1(n,r,d)=\min(\frac{\eta}{d},c^{-1})\le1$ so that 
\begin{align}
|R_{\overline{P}}^{^{\ge n}}(t)-R_{\overline{P}}^{^{\ge n}}(0)|\le \delta=\delta R_{\overline{P}}^{^{\ge n}}(0)~,\label{A.8M}
\end{align}
holds when $|t|\le C_1.$ 

Now given a polynomial $V\in E_r$ and $q\in \mathbb{R}^d$ define $$\overline{P}_q(t)=V(q+R_{V}^{^{\ge n}}(q)^{-1}t)~.$$ Writting \begin{align}
R_{\overline{P}_q}^{^{\ge n}}(t)=\sum\limits_{n\le |\alpha|\le r}|\partial^{\alpha}_t\overline{P}_{q}(t)|^{\frac{1}{|\alpha|}}=R_{V}^{^{\ge n}}(q)^{-1}R_{V}^{^{\ge n}}(q+R_{V}^{^{\ge n}}(q)^{-1}t)\;,\label{A.10M}
\end{align}clearly  the polynomial $\overline{P}_q$ belongs to the set $K.$ Hence for $\delta=\frac{1}{2}$ we get by (\ref{A.8M})
\begin{align}
\frac{1}{2}R_{\overline{P}_q}^{^{\ge n}}(0)\le R_{\overline{P}_q}^{^{\ge n}}(t)\le 2R_{\overline{P}_q}^{^{\ge n}}(0)~,\label{A.9M}
\end{align}
when $|t|\le C_1.$

It follows from (\ref{A.9M}) and (\ref{A.9M}),
\begin{align}
  \Big(\frac{R_{V}^{^{\ge n}}(q+R_{V}^{^{\ge n}}(q)^{-1}t)}{R_{V}^{^{\ge n}}(q)}\Big)^{\pm1}\le 2\;,\end{align}
for $|t|\le C_1.$

Therefore by the above inequality there is a constant $C_1\le1$ (chosen uniformly with respect to $q,V$ once $r, n$ and $d$ are fixed) so that for all $q,q'\in\mathbb{R}^d$ such that $R_{V}^{^{\ge n}}(q)|q'-q|\le C_1,$
\begin{align}
\Big(\frac{R_{V}^{^{\ge n}}(q')}{R_{V}^{^{\ge n}}(q)}\Big)^{\pm1}\le 2~.
\end{align}
It remains now to prove that for every $n\in\left\{1,...,r\right\}$, the metric $g^{n-1}$ is $g^n$ slow. Assuming the slowlness of $g^n$, the inequality \begin{align}
\Big(\frac{R_{ V}^{^{\ge n}}(q')}{R_{  V}^{^{\ge n}}(q)}\Big)^{\pm1}\le 2~.\label{A.13M}
\end{align} holds when $R_{V}^{^{\ge n}}(q)|q'-q|\le C_1^{-1}.$
 
Denote as before $t=R_{ V}^{^{\ge n}}(q)(q'-q)$ and $\overline{P}_q(t)= V(q+R_{ V}^{^{\ge n}}(q)^{-1}t).$ Taking into account (\ref{A.13M}) and (\ref{A.10M}) it results 
\begin{align*}
|t|\le C_1^{-1}\Longrightarrow|\partial_t^{\alpha}\overline{P}_q(t)|^{\frac{1}{|\alpha|}}\le R^{^{\ge n}}_{\overline{P}_q}(t)\le 2~,
\end{align*}
for all $\alpha\in\mathbb{N}^d$ with $|\alpha|\ge n.$
Consequently,
\begin{align}
|t|\le C_1^{-1}\Longrightarrow|\partial_t^{\alpha}\overline{P}_q(t)|\le 2^r~,\label{A.15M}
\end{align}
for all $\alpha\in\mathbb{N}^d,\;|\alpha|\ge n.$

Using (\ref{A.15M}), one has when $|\alpha|=n-1$ and $|t|\le C_1^{-1}$ 
\begin{align}
|\partial_t^{\alpha}\overline{P}_q(t)-\partial_t^{\alpha}\overline{P}_q(0)|&=|\sum\limits_{|\beta|=1}\frac{\partial_t^{\beta}[\partial_t^{\alpha}\overline{P}_q](0)}{\beta!}t^{\beta}|=|\sum\limits_{|\beta|=1}\frac{\partial_t^{\alpha+\beta}\overline{P}_q(0)}{\beta!}t^{\beta}|\nonumber\\&\le \sum\limits_{|\beta|=1}|\frac{\partial_t^{\alpha+\beta}\overline{P}_q(0)}{\beta!}|\;|t|\le d2^r|t|~,
\end{align}
since $|\alpha+\beta|\ge n.$
On the other hand, the application $$\mathbb{R}^{\mathbb{N}}\to \mathbb{R}^{\mathbb{N}},\; (u_{\beta})_{ |\beta|=n-1}\mapsto (|u_{\beta}|^{\frac{1}{|\beta|}})_{|\beta|=n-1}$$ is  continuous. Then for all $\delta>0$ there exists $\eta'=\eta'(n)>0$ so that
\begin{align}
\max\limits_{|\beta|=n-1}|u_{\beta}-v_{\beta}|\le \eta'\Longrightarrow \sum\limits_{|\beta|=n-1}\Big||u_{\beta}|^{\frac{1}{|\beta|}}-|v_{\beta}|^{\frac{1}{|\beta|}}\Big|\le \delta\;.
\end{align}
Hence for $\delta=1$ there is a strictly positive constant $C'_1=C'_1(n,r,d)=\min(\frac{\eta'}{d2^r},C_1^{-1})\le1$ so that 
\begin{align}
\Big|\sum\limits_{|\alpha|=n-1}|\partial_t^{\alpha}\overline{P}_q(t)|^{\frac{1}{|\alpha|}}-|\partial_t^{\alpha}\overline{P}_q(0)|^{\frac{1}{|\alpha|}}\Big|\le 1~,\label{A.18M}
\end{align}
holds when $|t|\le C'_1.$

Using respectively Peetre's inequality ( $\Big(\frac{\langle X'\rangle}{\langle X\rangle}\Big)^{s}\le 2^{\frac{|s|}{2}}\langle X-X'\rangle$ for all $s\in\mathbb{R},X,X'\in \mathbb{R}$) then (\ref{A.18M}) yields when $|t|\le C'_1$  \begin{align}
 \Big(\frac{\langle\sum\limits_{|\alpha|=n-1}|\partial_t^{\alpha}\overline{P}_q(t)|^{\frac{1}{|\alpha|}}\rangle}{\langle\sum\limits_{|\alpha|=n-1}|\partial_t^{\alpha}\overline{P}_q(0)|^{\frac{1}{|\alpha|}}\rangle}\Big)^{\pm1}\le \sqrt{2}\langle\sum\limits_{|\alpha|=n-1}|\partial_t^{\alpha}\overline{P}_q(t)|^{\frac{1}{|\alpha|}}-|\partial_t^{\alpha}\overline{P}_q(0)|^{\frac{1}{|\alpha|}}\rangle\le 2~.\label{A.18}
\end{align}
Remember that for any sequence $(a_i)_{1\le i\le N}$ of positive numbers
\begin{align}
\Big(\sum_1^{N}a_i^p\Big)^{\frac{1}{p}}\le \sum_1^{N}a_i\le N^{\frac{1}{q}}\Big(\sum_1^{N}a_i^p\Big)^{\frac{1}{p}}~,
\end{align}
where the two real numbers $p,q>1$ are conjugate indices. In particular for any  real numbers $a,b$
\begin{align*}
(a^2+b^2)^{\frac{1}{2}}\le (|a|+|b|)\le 2^2(a^2+b^2)^{\frac{1}{2}}
\end{align*}
It results from the elementary above inequality
\begin{align*}
\langle\sum\limits_{|\alpha|=n-1}|\partial_t^{\alpha}\overline{P}_q(t)|^{\frac{1}{|\alpha|}}\rangle\le (1+\sum\limits_{|\alpha|=n-1}|\partial_t^{\alpha}\overline{P}_q(t)|^{\frac{1}{|\alpha|}})\le 4\langle\sum\limits_{|\alpha|=n-1}|\partial_t^{\alpha}\overline{P}_q(t)|^{\frac{1}{|\alpha|}}\rangle~,
\end{align*}
and 
\begin{align*}
\langle\sum\limits_{|\alpha|=n-1}|\partial_t^{\alpha}\overline{P}_q(0)|^{\frac{1}{|\alpha|}}\rangle\le (1+\sum\limits_{|\alpha|=n-1}|\partial_t^{\alpha}\overline{P}_q(0)|^{\frac{1}{|\alpha|}})\le 4\langle\sum\limits_{|\alpha|=n-1}|\partial_t^{\alpha}\overline{P}_q(0)|^{\frac{1}{|\alpha|}}\rangle~.
\end{align*}
Using the above two estimates with (\ref{A.18}) we immediately get for $|t|\le C'_1$
\begin{align}
\Big(\frac{1+\sum\limits_{|\alpha|=n-1}|\partial_t^{\alpha}\overline{P}_q(t)|^{\frac{1}{|\alpha|}}}{1+\sum\limits_{|\alpha|=n-1}|\partial_t^{\alpha}\overline{P}_q(0)|^{\frac{1}{|\alpha|}}}\Big)^{\pm1}\le  8~.\label{A.20M}
\end{align}
Notice that by (\ref{A.13M})
\begin{align}
\Big(\frac{1+\sum\limits_{|\alpha|=n-1}|\partial_t^{\alpha}\overline{P}_q(t)|^{\frac{1}{|\alpha|}}}{1+\sum\limits_{|\alpha|=n-1}|\partial_t^{\alpha}\overline{P}_q(0)|^{\frac{1}{|\alpha|}}}\Big)^{\pm1}=\Big(\frac{1+\sum\limits_{|\alpha|=n-1}\Big(\frac{|\partial_q^{\alpha} V(q')|}{R_{ V}^{^{\ge n}}(q)^{|\alpha|}}\Big)^{\frac{1}{|\alpha|}}}{1+\sum\limits_{|\alpha|=n-1}\Big(\frac{|\partial_q^{\alpha} V(q)|}{R_{V}^{^{\ge n}}(q)^{|\alpha|}}\Big)^{\frac{1}{|\alpha|}}}\Big)^{\pm1}\ge \frac{1}{2}\Big(\frac{R_{V}^{^{\ge n-1}}(q')}{R_{V}^{^{\ge n-1}}(q)}\Big)^{\pm1}~.\label{A.21M}
\end{align}
In conclusion, from (\ref{A.20M}) and (\ref{A.21M}) there is a constant $C'_1=\min(\frac{\eta'}{d2^r},C_1^{-1})\le1$ so that
\begin{align}
R_{V}^{^{\ge n}}(q)|q-q'|\le C'_1\Longrightarrow \Big(\frac{R_{V}^{^{\ge n-1}}(q')}{R_{V}^{^{\ge n-1}}(q)}\Big)^{\pm1}\le 16\;.
\end{align}
\end{proof}
The main feature of a slow varying metric is that it is possible to introduce some partitions of unity related to the metric in a way made precise in the following theorem. For more details and proof see \cite{Hor1} ( Section 1.4 page 25).
 \begin{thm}\label{thm.B.6}\cite{Hor1}
 For any slowly varying metric $g$ in $\mathbb{R}^m$ one can choose a sequence $x_{\nu}\in \mathbb{R}^m$ such that the balls 
 \begin{align*}
 B_{\nu}=\left\{x;\;g_{x_{\nu}}(x-x_{\nu})<1\right\}
 \end{align*}
 form a covering of $\mathbb{R}^m$ for which the intersection of more than $N=(4C^3+1)^m$ balls $B_{\nu}$ is always empty ($C$ is the constant in (\ref{A.1})). In addition, for any decreasing sequence $d_i$ with $\sum\limits_j d_j=1$ one can choose non negative $\phi_{\nu}\in\mathcal{C}_0^{\infty}(B_{\nu})$ with $\sum \phi_{\nu}=1$ in $\mathbb{R}^m$ so that for all $k$
 \begin{align*}
 |\phi_{\nu}^{(k)}(x;y_1,\cdots,y_k)|\le (NCC_1)^kg_x(y_1,0)\cdots g_x(y_k,0)/d_1\cdots d_k
 \end{align*}
 where $C$ is the constant in (\ref{A.1}) and $C_1$ is a constant that depends only on $m.$
 \end{thm}
 Regarding the above Theorem we have the following result.
\begin{lem}\label{lemA.6}
Let $P\in E_r$ and $n\in\left\{1,...,r\right\}$, then there exists a partition of unity $\displaystyle\sum\limits_{j\in \mathbb{N}}\psi_j(q)^2\equiv~1$ in $\mathbb{R}^d$ such that:
\\
1) For all $ q\in \mathbb{R}^d,$ the cardinality of the set $\lbrace j, \Psi_j(q)\neq 0\rbrace$ is uniformely bounded.
\\
2) For any natural number $j\in\mathbb{N}, $ \begin{align*}\displaystyle\mathrm{supp}\;\Psi_j \subset B(q_j,aR^{^{\ge n}}_P(q_j)^{-1})\quad\text{and} \quad\;\Psi_j\equiv 1\;\;\text{in}\; B(q_j,bR^{^{\ge n}}_P(q_j)^{-1})~,\end{align*} for some $q_j\in \mathbb{R}^d$ with $0<b<a$ independent of $j\in \mathbb{N}~.$
\\
3) For all $\alpha\in\mathbb{N}^d\setminus\lbrace 0\rbrace,$ there exists $c_{\alpha}>0$ such that 
$$\displaystyle\sum\limits_{j\in\mathbb{N}}|\partial_q^{\alpha}\Psi_j|^2\le c_{\alpha} R^{^{\ge n}}_P(q)^{2|\alpha|}~.$$

Moreover the constants $a,b \;\text{et}\ c_{\alpha}$ can be chosen uniformly with respect to $P\in E_r$, once the degree $r\in \mathbb{N}\;\text{ and the dimension }\; \ d\in \mathbb{N} \;\text{are fixed}.$
\end{lem}

\end{appendices}
\begin{appendices}
\section{Around Tarski-Seidenberg theorem}\label{B}
In this appendix we give an application of the Tarski-Seidemberg theorem \cite{Hor2}, which we state in the following geometric form. We first introduce a few basic concepts which are needed for the state.
\begin{defi}
A subset of $\mathbb{R}^n$ is called semi-algebraic if it is a finite union of finite intersections of sets defined by polynomial equations or inequalities.
\end{defi}
\begin{defi}
Let $A \subset \mathbb R^n$ and $B \subset \mathbb R^m$ be two sub-algebraic sets. The function $f : A\to B$ is said to be semi-algebraic if its graph
$\Gamma_f = \left\{(x, y) \in A \times B ;\; y = f(x)\right\}$
is a semi-algebraic set of  $\mathbb R^n \times \mathbb{R}^m.$
\end{defi}
\begin{thm}\label{th.C.3}\cite{Hor2}(Tarski-Seidenberg)
If $A$ is a semi-algebraic subset of $\mathbb{R}^{n+m}=\mathbb{R}^{n}\oplus \mathbb{R}^{m}$, then the projection $A'$ of $A$ in $\mathbb{R}^{m}$ is also semi-algebraic.
\end{thm}

\begin{prop}\label{prop.C.3.1}\cite{Hor2}\label{propB.6.1}
If $E$ is a semi-algebraic set on $\mathbb{R}^{2+n}$, and 
\begin{align*}f(x)=\inf\left\{y\in\mathbb{R};\;\exists z\in \mathbb{R}^n, (x,y,z)\in E\right\}\end{align*} is defined and finite for large positive $x$, then $f$ is identically  0 for lage $x$ or else
\begin{align*}f(x)=Ax^a(1+o(1))~,\;\;\;x\to +\infty\end{align*}
where $A\not=0$ and $a$ is a rational number.
\end{prop}
We refer to \cite{Hor2} (see Theorem A.2.2 and Theorem A.2.5) for detailed proofs of Theorem~\ref{th.C.3} and Proposition~\ref{prop.C.3.1}.

In the final part of this section we list and recall the following notations. 
\begin{NT} Let $P$ be a polynomial of degree $r.$ For all natural number $n\in\left\{0,\cdots ,r\right\}$ and every $q\in \mathbb{R}^d$ 
 \begin{align}
R^{^{\ge n}}_P(q)=\sum\limits_{ n\le|\alpha|\le r}|\partial^{\alpha}_qP(q)|^{\frac{1}{|\alpha|}}\;,\label{A.43}
 \end{align}
  \begin{align}
R^{= n}_P(q)=\sum\limits_{ |\alpha|=n}|\partial^{\alpha}_qP(q)|^{\frac{1}{|\alpha|}}\;.\label{A.44}
 \end{align}
\end{NT}
\begin{lem}\label{lem2.1}
Let $\Sigma$ be an unbounded semialgebraic set and $V$ a polynomial of degree $r$ satisfying the following assumption
\begin{align}
\lim\limits_{\substack{q\to\infty \\ q\in \Sigma}}\frac{R_V^{^{\ge n}}(q)^{\alpha}}{R_V^{= m}(q)^2}=0~,
\end{align}
where $\alpha\in\mathbb{Q}, n, m\in\left\{0,1,\cdots ,r-1\right\}$ are fixed numbers. 

Then there exist $\delta\in (0,1)$ and a positive function
$\Lambda_{\Sigma}: (0,+\infty)\to [0,+\infty)$ so that 
\begin{eqnarray*}
  &&\forall q\in \Sigma\,, |q|\geq \varrho\,, \quad
     \Lambda_{\Sigma}(\varrho)R_V^{^{\ge n}}(q)^{\alpha}\le
     \,R_V^{=m}(q)^{2(1-\delta)}\label{B.1}\\
\text{and}&&
\lim_{\varrho\to +\infty}\Lambda_{\Sigma}(\varrho)=+\infty\,.
\end{eqnarray*}
\end{lem}

\begin{proof}
Suppose that there are $\alpha\in\mathbb{Q}, n, m\in\left\{0,1,\cdots ,r-1\right\}$ such that 
\begin{align}
\lim\limits_{\substack{q\to\infty \\ q\in \Sigma}}\frac{R_V^{^{\ge n}}(q)^{\alpha}}{R_V^{= m}(q)^2}=0~,\label{B.55}
\end{align}
where $\Sigma$ is a given unbounded semialgebraic set. 

After setting $\tau=\mathrm{ppcm}\Big(|\beta|,\;\min(n,m)\le |\beta|\le r\Big)~,$
define the functions $\widetilde{R}_V^{^{\ge n}}$ and $\widetilde{R}_V^{= m}~,$ for all $q\in\mathbb{R}^d$ by $$\widetilde{R}_V^{^{\ge n}}(q)=\sum\limits_{n\le |\alpha|\le r}|\partial^{\alpha}_qV(q)|^{\frac{\tau}{|\alpha|}}$$ and $$\widetilde{R}_V^{= m}(q)=\sum\limits_{ |\alpha|= m}|\partial^{\alpha}_qV(q)|^{\frac{\tau}{|\alpha|}}~.$$ 
Notice that one has the equivalences $R_V^{^{\ge n}}(q)\asymp \Big(\widetilde{R}_V^{^{\ge n}}(q)\Big)^{\frac{1}{\tau}}$  and $R_V^{=m}(q)\asymp\Big(\widetilde{R}_V^{=m}(q)\Big)^{\frac{1}{\tau}}$ for all $q\in\mathbb{R}^d$ where the functions $R_V^{^{\ge n}}$ and $R_V^{=m}$ are defined respectively as in (\ref{A.43}) and (\ref{A.44}). Clearly the Assumption (\ref{B.55}) is equivalent to 
\begin{align}
\lim\limits_{\substack{q\to\infty \\ q\in \Sigma}}\frac{\widetilde{R}_V^{^{\ge n}}(q)^{\alpha}}{\widetilde{R}_V^{= m}(q)^2}=0\;.\label{B.3}
\end{align}
Remark here that $\widetilde{R}_V^{^{\ge n}}(q)$ and $\widetilde{R}_V^{= m}(q)$ are polynomials in $q\in\mathbb{R}^d$ variable. Furthermore, the Assumption (\ref{B.3}) can be written as follows \begin{align*}
\widetilde{R}_V^{^{\ge n}}(q)^{\alpha}\le \epsilon(q)\widetilde{R}_V^{=m}(q)^2\;,
\end{align*}
for all $q\in \Sigma$ where \begin{align}\epsilon(q)=\inf\left\{\epsilon>0,\;\epsilon\widetilde{R}_V^{=m}(q)^2-\widetilde{R}_V^{^{\ge n}}(q)^{\alpha}>0\right\}~,\quad\quad
\lim\limits_{\substack{q\to\infty \\ q \in  \Sigma}} \epsilon(q)=0~.\label{B.8M}
\end{align}
Now, following the notations of Proposition~\ref{prop.C.3.1}, we
introduce the
set \begin{align*}E=\left\{(q,\varrho,\epsilon)\in\mathbb{R}^{d+2}\;
      \text{such
      that}\;\epsilon\widetilde{R}_V^{=m}(q)^2-\widetilde{R}_V^{^{\ge
      n}}(q)^{\alpha}>0\;\text{and}\;|q|^2\ge
      \varrho^2\right\}~,\end{align*}
 and the function $f$ defined in
  $\mathbb{R}_+$ by
 \begin{align}f(\varrho)=\inf\left\{\epsilon>0,\; \text{such that}\; (q,\varrho,\epsilon)\in E\right\}~.\label{B.9M}\end{align}
By Tarski-Seidenberg theorem (see Theorem~\ref{th.C.3}), the function
$f$ is semialgebraic in $\varrho.$ Moreover $f$ is defined, finite and not
identically zero. Then by Proposition~\ref{propB.6.1}, there exist a constant $A>0$ and a rational number $\gamma$ such that 
\begin{align*}
f(\varrho)=A\varrho^{\gamma}+o_{\varrho\to +\infty}(\varrho^{\gamma})~.
\end{align*} 
By the definition (\ref{B.9M}) and (\ref{B.8M}), $\lim\limits_{\varrho\to
  +\infty}f(\varrho)=0$ and then $\gamma<0$\,. Hence for  $\varrho\geq
1$\,,  we
know $f(\varrho)\leq \frac{2A}{\varrho^{|\gamma|}}$\,.
We deduce for $|q|\geq 1$\,,
\begin{align}
\label{B.6}
\widetilde{R}_V^{^{\ge n}}(q)^{\alpha}\le
 f(|q|)\widetilde{R}_V^{=m}(q)^2
\le \frac{2A}{|q|^{|\gamma|}}\widetilde{R}_V^{=m}(q)^2
\end{align}
and
\begin{equation}
  \label{B.6.2}
\frac{|q|^{|\gamma|/2}}{2A}\widetilde{R}_{V}^{\geq n}(q)^{\alpha}\leq
\frac{1}{|q| ^{\frac{|\gamma|}{2}}}\widetilde{R}_{V}^{=m}(q)^{2}\,.
\end{equation}
In particular, since $\widetilde{R}_V^{^{\ge n}}(q)\ge
\widetilde{R}_V^{=r}(0)>0$\,, $\widetilde{R}_V^{=m}(q)$ does not
vanish for $q\in \Sigma$ with $|q|\ge 1.$ 

On the other hand, notice
\begin{align}
\forall q\in \Sigma\,,\, |q|\geq 1\,,\quad
\widetilde{R}_V^{=m}(q)\le c|q|^{\tau r}\,.
\label{B.11M}\end{align}
The inequalities  (\ref{B.6}) and (\ref{B.11M}) lead to 
\begin{align*}
\widetilde{R}_V^{^{\ge n}}(q)^{\alpha}\le C|q|^{2\tau r-|\gamma|}
\end{align*}
 for every $q\in \Sigma$ with $|q|\ge \rho\ge1.$ Therefore
 since $\widetilde{R}_V^{^{\ge n}}(q)\geq \widetilde{R}_{V}^{=r}(0)>0$
 we deduce $|\gamma|\leq 2\tau r.$
 \\
Using again (\ref{B.11M}) we get
\begin{align}
\frac{1}{|q|^{\frac{|\gamma|}{2}}}\le \frac{c^{\frac{|\gamma|}{2\tau r}}}
{\widetilde{R}_V^{=m}(q)^{\frac{|\gamma|}{2 \tau r}}}~,\label{B.7}\end{align}
for any $q\in \Sigma$  with $|q|\geq 1.$ 

From (\ref{B.6.2}) and (\ref{B.7}),  we deduce
\begin{align}
\forall q\in \Sigma\,, |q|\geq \varrho\geq 1\,,\quad
\frac{\varrho^{|\gamma|/2}}{2A}
\widetilde{R}_V^{^{\ge n}}(q)^{\alpha}\leq
\frac{|q|^{|\gamma|/2}}{2A}
\widetilde{R}_V^{^{\ge n}}(q)^{\alpha}\leq 
c^{\frac{|\gamma|}{2\tau r}}\,\widetilde{R}_V^{=m}(q)^{2(1-\frac{|\gamma|}{4\tau r})}\;.
\end{align}
We now take $\delta=\frac{|\gamma|}{4\tau r}\in (0,1)$ and
$$
\Lambda_{\Sigma}(\varrho)=\left\{
  \begin{array}[c]{ll}
    \frac{\varrho^{|\gamma|/2}}{2Ac^{\frac{|\gamma|}{2\tau
    r}}}&\text{if}\quad \varrho\geq 1\\
0&\text{else}\,.
  \end{array}
\right.
$$
\end{proof}

\end{appendices}
\textbf{Acknowledgement} I express my sincere gratitude to Professor Francis Nier. As a PhD advisor, Professor Nier supported me in this work.

\end{document}